\documentclass[final,onefignum,onetabnum]{siamart171218}
\pdfoutput=1
\usepackage{subfigure}

\usepackage{graphicx}
\usepackage{url}
\usepackage{bm}
\usepackage{paralist}

\usepackage{algorithm}
\usepackage{algorithmic}        
\usepackage{multirow}
\renewcommand{\algorithmicrequire}{\textbf{Initialization:}}   
\renewcommand{\algorithmicensure}{\textbf{Iteration:}}
\renewcommand{\algorithmiclastcon}{\textbf{Output:}} 

\newcommand{\e}{\begin{equation}}
\newcommand{\ee}{\end{equation}}
\newcommand{\en}{\begin{equation*}}
\newcommand{\een}{\end{equation*}}
\newcommand{\eqn}{\begin{eqnarray}}
\newcommand{\eeqn}{\end{eqnarray}}
\newcommand{\bmat}{\begin{bmatrix}}
\newcommand{\emat}{\end{bmatrix}}
\newcommand{\BIT}{\begin{itemize}}
\newcommand{\EIT}{\end{itemize}}
\newcommand{\triangleqtemp}{\stackrel{\bigtriangleup}{=}}

\newcommand{\argmin}{\mathop{\rm argmin}}

\newcommand{\va}{\bm a}

\newcommand{\vd}{\bm d}
\newcommand{\ve}{\bm e}
\newcommand{\vf}{\bm f}

\newcommand{\vr}{\bm r}

\newcommand{\vu}{\bm u}
\newcommand{\vv}{\bm v}

\newcommand{\vx}{\bm x}

\newcommand{\valpha}{\bm \alpha}

\newcommand{\vdelta}{\bm\delta}

\newcommand{\vtheta}{\bm \theta}

\newcommand{\mA}{\bm A}

\newcommand{\mE}{\bm E}

\newcommand{\mI}{\bm I}

\newcommand{\mM}{\bm M}

\newcommand{\mP}{\bm P}

\newcommand{\mUpsilon}{\bm \Upsilon}

\newcounter{oursection}

\usepackage{soul}

\usepackage{lipsum}
\usepackage{amsfonts}
\usepackage{graphicx}
\usepackage{epstopdf}
\usepackage{algorithmic}
\ifpdf
  \DeclareGraphicsExtensions{.eps,.pdf,.png,.jpg}
\else
  \DeclareGraphicsExtensions{.eps}
\fi

\newsiamremark{remark}{Remark}
\newsiamremark{hypothesis}{Hypothesis}
\crefname{hypothesis}{Hypothesis}{Hypotheses}
\newsiamthm{claim}{Claim}

\headers{Merging Multigrid Optimization with SESOP}{T. Hong \& I. Yavneh \& M. Zibulevsky}

\title{Merging Multigrid Optimization with SESOP\thanks{Submitted to the editors DATE.
\funding{This work was funded by ****}}}

\author{Tao Hong, Irad Yavneh, and Michael Zibulevsky}

\usepackage{amsopn}

\usepackage{subfigure}
\usepackage{diagbox}
\renewcommand\qedsymbol{\hfill $\square$}
\usepackage{tikz}
\usetikzlibrary{arrows,shapes,snakes,automata,backgrounds,petri}
\usepackage[latin1]{inputenc}
\ifpdf
\hypersetup{
  pdftitle={Merging Multigrid Optimization with SESOP},
  pdfauthor={T. Hong, I. Yavneh, and M. Zibulevsky}
}
\fi

\begin{document}
\maketitle
\vspace*{0.1in} 
{
\begin{abstract}
A merger of two optimization frameworks is introduced: SEquential Subspace OPtimization (SESOP) with MultiGrid (MG) optimization. At each iteration of the algorithm, the search direction implied by the coarse-grid correction process of MG is added to the low dimensional search-space of SESOP, which includes the preconditioned gradient and search directions involving the previous iterates, called {\em history}. Numerical experiments demonstrate the effectiveness of this approach. We then study the asymptotic convergence factor of the two-level version of SESOP-MG (dubbed SESOP-TG) for optimization of quadratic functions, and derive approximately optimal fixed parameters, which may reduce the computational overhead for such problems significantly. 
\end{abstract}

\begin{keywords}
 SESOP, acceleration, multigrid, optimization, linear and nonlinear large-scale problems
\end{keywords}

\begin{AMS}
  90C06, 90C25, 65N55, 65K05, 65N22, 65H10
\end{AMS}

\section{Introduction}
\label{sec:intro}
Multigrid (MG) methods are widely considered to be an efficient approach for solving elliptic partial differential equations (PDEs) and systems, as well as other problems which can be effectively represented on a hierarchy of grids or levels \cite{brandt1977multi,briggs2000multigrid,stuben2001introduction,xu2017algebraic}. However, it is often challenging to design efficient stand-alone MG methods for difficult problems, and therefore MG methods are often used in combination with acceleration techniques (e.g., \cite{oosterlee2000krylov}). 

In this paper we study MG in an optimization framework and seek robust solution methods by merging this approach with so-called SEquential Subspace OPtimization (SESOP) \cite{zibulevsky2013speeding}. SESOP is a general framework for iteratively solving large-scale optimization problems, as described in the next section. The combined  framework is called SESOP-MG, and its two-grid (TG) version is called SESOP-TG. 
We then analyze the asymptotic convergence factor (ACF) of a fixed-stepsize version of SESOP-TG for quadratic optimization problems, and estimate the expected acceleration due to SESOP by means of the so-called $h$-ellipticity measure \cite{Bra84,trottenberg2000multigrid}. Resorting to local Fourier analysis (LFA) \cite{brandt2011multigrid,wienands2004practical}, we propose two  methods to estimate optimal fixed stepsizes of SESOP-TG for quadratic optimization problems cheaply in cases where LFA is applicable. Numerical tests demonstrate the relevance of the theoretical analysis in practice.

The paper is organized as follows. The standard TG and SESOP algorithms are briefly described in the remainder of this section. The merged SESOP-TG/MG algorithm is proposed and tested in \Cref{sec:SESOPMGMerge}. An analysis of the ACF for a fixed-stepsize version of the SESOP-TG method for quadratic problems is presented in \Cref{sec:ConvFactLinear}. There, we also show how to estimate the optimal fixed stepsizes cheaply for some specific problems. Numerical tests validate our analysis and the effectiveness of the proposed method for estimating the optimal fixed stepsizes. Conclusions are drawn in \Cref{sec:conclusion}.


We adopt the following notation. $\vx$ denotes the unknown solution vector, and $F(\vx)$ a convex function we aim to minimize. In the two-grid case, we use superscripts $h$ and $H$ to denote the fine and coarse grid, e.g., $F^h(\vx^h)$ and $F^H(\vx^H)$ denote the fine and coarse functions, respectively. 
In general, we use boldface font to denote vectors and matrices, and $^\mathcal T$ denotes the transpose operator.  


\subsection{Multigrid (MG)}\label{sec:subsec:pre:MG}

We consider MG as a method for convex optimization. For an extensive review of MG for PDE optimization, see Borz\`{i} and Schulz \cite{borzi2009multigrid} and references therein. Several authors developed MG optimization for specific problems in the 1990's, and at the end of the decade Nash \cite{nash2000multigrid} formulated MG as a general optimization framework called MG/OPT based directly on the well-known full approximation storage (FAS) scheme of Brandt \cite{brandt1977multi}. Following this, Lewis and Nash applied this framework to systems governed by differential equations \cite{lewis2005model}. In similar vein, Wen and Goldfard proposed a line search MG method to solve unconstrained convex and nonconvex problems \cite{wen2009line}. Toint et al. merged MG optimization with the trust region approach, applying MG to the series of linear subproblems arising in each step of trust region methods. This was applied to nonlinear convex or nonconvex problems, including bound constraints \cite{gratton2008recursiveIMA,gratton2008recursive,toint2009multilevel,gratton2010numerical}. Recently, Calandra et al.  applied MG optimization to reduce the cost of step computation in high-order optimization \cite{calandra2021high}.

Consider the following unconstrained problem defined on the fine-grid: 
\e 
\vx_*^h = \argmin_{\vx^h\in\mathbb{R}^N} F^h(\vx^h) \, , \label{eq:uncons:fine:TG} 
\ee
where $ F^h(\vx^h): \mathbb{R}^N\rightarrow \mathbb R$ is a convex and differentiable function and the solution set of \eqref{eq:uncons:fine:TG} is not empty. We describe a single iteration of the two-grid algorithm next, with MG obtained by straightforward recursion. Denote by $\vx_k^h$ the approximation to the fine-grid solution $\vx_*^h$ after the $k$th iteration. Assume that we have defined two full-rank operators, a {\em restriction} $\mI_h^H: \mathbb{R}^N \rightarrow \mathbb{R}^{N_c}$, and a {\em prolongation} $\mI_H^h: \mathbb{R}^{N_c} \rightarrow \mathbb{R}^N$, where $N_c$ is the size of the coarse-grid.  Furthermore, let $\vx_k^H$ denote a coarse-grid approximation to $\vx_k^h$ (for example, we may use $\vx_k^H = \mI_h^H \vx_k^h$, but other choices may be used as well). The coarse-grid 
problem is then defined as follows:
\e
\vx_*^{H} = \argmin_{\vx^H\in\mathbb{R}^{N_c}} F^H(\vx^H)-\vv_k^\mathcal T\vx^H,\label{eq:coarse:TG}
\ee
where $F^H$ is a coarse approximation to $F^h$, and $\vv_k=\nabla F^H(\vx^H_k)-\mI_h^H \nabla F^h(\vx^h_k)$. The correction term $\vv_k$, adapted from FAS, is used to enforce the same  
first-order optimality condition \cite{nash2000multigrid} on the fine and coarse levels. After solving \eqref{eq:coarse:TG} (exactly in TG, and approximately and recursively in MG), the coarse-grid correction CGC direction is defined by
\e
\vd_k^h = \mI_H^h (\vx_*^H - \vx_k^H). \label{eq:CGC}
\ee
Finally, the CGC is added to the current fine-grid approximation:
\e
\vx^h_k \leftarrow \vx^h_k + \vd_k^h.\label{eq:CtoFC}
\ee
This may be followed by additional relaxation steps, yielding the $k+1$st approximation $\vx^h_{k+1}$. The two-grid algorithm is written in \Cref{alg:TG}.

\begin{algorithm}
\caption{Two-Grid (TG)}
\label{alg:TG}
\begin{algorithmic}[1]
\REQUIRE {Initial value $\vx_0^h$, convergence criterion $\epsilon$, the number of maximal iterations $\text{Max\_Iter}$, $\nu_1, \nu_2$ -- the number of pre- and post-relaxation steps, $k=1$.}
\lastcon {Solution $\vx_*^h.$}
\WHILE {$k \leq {Max\_Iter}$}
\STATE $\vx_k^h\leftarrow Relaxation(\vx_{k-1}^h,\nu_1)$.
\STATE {Evaluate the gradient $\nabla  F^h(\vx_k^h)$ and formulate $\vv_k$.}
\IF {$|\nabla  F^h(\vx_k^h)| \leq \epsilon$}
\STATE $\vx_*^h \leftarrow \vx_k^h$.
\RETURN
\ENDIF
\STATE Solve the coarse problem \eqref{eq:coarse:TG} to get $\vx_*^H$ and then $\vd_k^h\leftarrow\mI_H^h\left(\vx_*^H-\mI_h^H \vx_k^h\right)$. 
\STATE Update $\vx_k^h \leftarrow \vx_k^h+\vd_k^h$.\label{alg:TG:CGC}
\STATE $\vx_k^h\leftarrow Relaxation(\vx_k^h,\nu_2)$.
\STATE $k \leftarrow k+1$
\ENDWHILE
\STATE $\vx_*^h \leftarrow \vx_k^h$.
\RETURN 
\end{algorithmic}
\end{algorithm}


\subsection{SEquential Subspace OPtimization (SESOP)}\label{sec:subsec:pre:sesop}
SESOP \cite{narkiss2005sequential,elad2007coordinate,zibulevsky2010l1,zibulevsky2013speeding} is a framework for solving smooth large-scale unconstrained minimization problems such as \eqref{eq:uncons:fine:TG},
by sequential optimization over affine subspaces ${\cal{\bm M}}_k$, spanned by the current descent direction (typically preconditioned gradient) and $m$ previous propagation directions of the method. If $F^h(\vx^h)$ is convex, SESOP yields the optimal worst-case convergence factor of $\mathcal O(\frac{1}{k^2})$ \cite{narkiss2005sequential}, while achieving efficiency of the quadratic Conjugate Gradients (CG) method when the problem is close to quadratic or in the vicinity of the solution.

The affine subspace at the $k$th iteration is defined by
\en
{\cal {\bm M}}_k^h = \left\{\bm x_k^h+\bm P_k^h\bm \alpha: \bm\alpha\in \mathbb{R}^{m+1}  \right\},
\een
where $\vx_k^h$ is the $k$th iterate, the matrix $\bm P_k^h$ contains the spanning directions in its columns, the preconditioned gradient $\mathcal {\mP}\nabla F^h(\vx_k^h)$ and $m$ last steps $\vdelta_i^h = \vx_i^h - \vx_{i-1}^h$,
\e
\bm P_k^h=\left[\mathcal {\mP}\nabla F^h(\vx_k^h), \vdelta_{k}^h, \vdelta_{k-1}^h,\cdots,\vdelta_{k-m+1}^h \right],
~~ m\geq 0.\label{eq:sesop_affine_def1}
\ee
The new iterate $\bm x_{k+1}^h$ is obtained via optimization of  $F^h(\cdot)$ over the current subspace ${\cal{\bm M}}_k^h$,
\e
\left.\begin{array}{rcl}
\bm \alpha_k^*&=&\argmin_{\bm\alpha\in \mathbb R^{m+1}} F^h(\bm x_k^h+\bm 
P_k^h\bm\alpha),\\
\bm x_{k+1}^h&=&\bm x_k^h+\bm P_k^h \bm\alpha_k^*.\label{eq:SESOP:affine:optimize}
\end{array}\right.
\ee
By keeping the dimension of $\valpha$ low,  \eqref{eq:SESOP:affine:optimize} can be solved efficiently with a Newton-type method. SESOP is relatively efficient if $(i)$ the objective function can be represented as $F^h(\mathcal\mA\vx^h)$; $(ii)$ the evaluation of $F^h(\cdot)$ is cheap, but calculating  $\mathcal\mA\vx^h$ is expensive.
%
%
In this case, we can avoid re-computation of  $\mathcal\mA\vx_k^h$ during the subspace minimization by storing $\mathcal\mA\mP_k^h$ and $\mathcal\mA\vx_k^h$.
A typical class of problems of this type is the well-known $\ell_1-\ell_2$ optimization \cite{zibulevsky2010l1}.

SESOP can yield faster convergence if more efficient descent directions are added to the subspace ${\cal{\mM}}_k^h$. This may include a cumulative or parallel coordinate descent step, a separable surrogate function or expectation-maximization step, Newton-type steps and  other methods \cite{elad2007coordinate,zibulevsky2010l1,zibulevsky2013speeding}. In this work we suggest adding the CGC direction provided by the MG framework.


\section{Merging multigrid with SESOP}
\label{sec:SESOPMGMerge}

\subsection{SESOP-TG} \label{sec:sub:SESOPTG}
We begin this section by introducing a two-grid version of our scheme, SESOP-TG, and later extend it to the multilevel version, SESOP-MG.  
In its basic form, the idea is to add the CGC $\vd_k^h$ of \eqref{eq:CGC} into the affine subspace ${\cal{\bm M}}_k$ by replacing $\mP_k^h$ with an augmented subspace $\bar{\mP}_k^h=\bmat \vd_k^h&\mP_k^h\emat$ and computing the locally optimal $\valpha$ to obtain the next iterate, $\vx_{k+1}^h$. Our intention is to combine the efficiency of MG that results from fast reduction of smooth error by the CGC, with the robustness of SESOP that results from the local optimization over the subspace. That is, even in difficult cases where the CGC is inadequate, the algorithm should still converge at least as fast as the standard SESOP, 
because an inefficient direction simply results in small (or conceivably even negative) coefficient. SESOP-TG is presented in \Cref{alg:SESOP-TG}. Note that we add the option of two additional steps to the usual SESOP algorithm, the pre- and post-relaxation at steps \ref{alg:SESOP-TG:pre} and \ref{alg:SESOP-TG:post} in \Cref{alg:SESOP-TG}, commonly applied in MG algorithms. This allows us to advance the solution with low computational expense whenever the coarse-grid direction is inefficient (due the the fact that previous use of the CGC direction greatly reduced the smooth error). For the relaxation we typically use nonlinear Jacobi or Gauss-Seidel applied to the gradient equation. The flowchart of \Cref{alg:SESOP-TG} is presented in \Cref{fig:flowchart:SESOP-TG-m}.

\begin{algorithm}[!htb]
\caption{SESOP-TG-$m$}
\label{alg:SESOP-TG}
\begin{algorithmic}[1]
\REQUIRE {Initial value $\bar\vx_0^h=\vx_0^h$, convergence criterion $\epsilon$, the number of maximal iterations $\text{Max\_Iter}$, the preconditioner $\mathcal {\mP}$ and $\nu_1, \nu_2$ -- the number of pre- and post-relaxation steps, $k=1$.}
\lastcon {$\vx_*^h.$}
\WHILE {$k \leq {Max\_Iter}$}
\STATE $ \bar{\bar\vx}_k^h\leftarrow Relaxation(\bar\vx_{k-1}^h,\nu_1)$.\label{alg:SESOP-TG:pre}
\STATE Evaluate the gradient $\nabla  F^h(\bar{\bar\vx}_k^h)$. 
\IF {$|\nabla  F^h(\bar{\bar\vx}_k^h)| \leq \epsilon$}
\STATE $\vx_*^h \leftarrow \bar{\bar\vx}_k^h$.
\RETURN
\ELSE
\STATE $\mP_k^h \leftarrow \bmat \mathcal {\mP}\nabla  F^h(\bar{\bar\vx}_k^h)&\bar{\bar\vx}_{k}^h-\vx_{k-2}^h&\cdots& \bar{\bar\vx}_{k-m+1}^h-\vx_{k-m-1}^h\emat$.
\ENDIF
\STATE Solve \eqref{eq:coarse:TG} with initial $\vx_k^H$ to get $\vx_*^H$ and then $\vd_k^h\leftarrow\mI_H^h\left(\vx_*^H-\vx_k^H\right)$.  \label{alg:SESOP-TG:constr:solve_coarse}
\STATE  Formulate the augmented $\bar{\mP}_k^h\leftarrow \bmat \vd_k^h &\mP_k^h\emat$.
\STATE Solve approximately \eqref{eq:SESOP:affine:optimize} on $\bar{\mP}_k^h$ and update $\vx_k^h \leftarrow \bar{\bar\vx}_k^h+\bar{\mP}_k^h\valpha_k^*$.\label{alg:SESOP-TG:constr:SESOP}
\STATE $\bar\vx_k^h\leftarrow Relaxation(\vx_k^h,\nu_2)$.\label{alg:SESOP-TG:post}
\ENDWHILE
\STATE $\vx_*^h \leftarrow \bar \vx_k^h$.
\RETURN 
\end{algorithmic}
\end{algorithm}

\begin{figure}[!htb]
\centering
\begin{tikzpicture}[node distance=2cm,auto,>=latex']
\tikzstyle{int}=[draw, fill=blue!20, minimum size=2em]
\tikzstyle{init} = [pin edge={to-,thin,black}]
\tikzstyle{line} = [draw,thick,-latex]
\tikzstyle{transition} = [font=\small]
    \node [int] (a) {$\vx_{k-2}^h$};
    \node (b) [left of=a,node distance=1.5cm, coordinate] {a};
    \node [int] (c) [above right of=a] {$\bar \vx_{k-2}^h$};
    \node [int] (d) [below right of=c] {$\bar{\bar \vx}_{k-1}^h$};
    \node [int] (e) [right of = d] {$\vx_{k-1}^h$};
    \node [int] (f) [above right of=e] {$ \bar\vx_{k-1}^h$};
    \node [int] (g) [below right of=f] {$\bar{\bar \vx}_k^h$};
    \node [int] (h) [right of=g] {$\vx_k^h$}; 
    \node [coordinate] (end) [right of=h, node distance=1.5cm]{};
     
    \path [dashed,line] (b) -- (a) node {};
    \path [line] (a) -- (c) node[transition,pos=0.1,above,align=left] {$\nu_2$};
     \path [line] (c) -- (d) node[transition,pos=0.8,above,align=left] {$\nu_1$};
     \path [line] (d) -- (e) node[font=\scriptsize,pos=0.5,above,align=left] {step \ref{alg:SESOP-TG:constr:SESOP}};
     \path [line] (e) -- (f) node[transition,pos=0.1,above,align=left] {$\nu_2$};
     \path [line] (f) -- (g) node[transition,pos=0.8,above,align=left] {$\nu_1$};
     \path [line] (g) -- (h) node[font=\scriptsize,pos=0.5,above,align=left] {step \ref{alg:SESOP-TG:constr:SESOP}};
     \draw [<->,thick] (g) -- ++(0,-1) -|  (a) node[font=\scriptsize,pos=0.23,above,align=right] {history};
	\path [dashed,line] (h) edge node {} (end) ;
		\draw [<->,thick] (f) -- ++(0,1) -|  (c) node[font=\scriptsize,pos=0.23,above,align=right] {$(k-1)$st iteration};
\end{tikzpicture}

\caption{Flowchart of the iterates of \Cref{alg:SESOP-TG}.}
	\label{fig:flowchart:SESOP-TG-m}
\end{figure}

\begin{remark}\label{remark:alg:SESOP-TG}

Evidently, if we select $\bar{\mP}_k^h$ to contain only $\vd_k^h$, with $\bm\alpha_k^*=1$, then SESOP-TG-$m$ reduces to TG. Note that in \Cref{alg:TG}, we need to valuate $\nabla F^h(\vx_k^h)$ to formulate $\vv_k$ but we do not use it further. Here, however, we put $\nabla F^h(\vx_k^h)$ together with the CGC direction in the augmented subspace $\bar{\mP}_k^h$ for further use. The main drawback of \Cref{alg:SESOP-TG} is of course the need to solve the subspace minimization problem for $\bm\alpha_k^*$ at each iteration. However, in certain cases of problems with special structure or large size, the cost of the subspace minimization may become negligible; see numerical tests in \Cref{sec:sub:NumericalTestsI}.

\end{remark}

\begin{remark}
\label{remark:alg:SESOP-MG}
	Similarly to standard multigrid methods, the multilevel algorithm called SESOP-MG-$m$ is derived by recursively treating \eqref{eq:coarse:TG} at step \ref{alg:SESOP-TG:constr:solve_coarse} of \Cref{alg:SESOP-TG}. Compared with the subspace $\bar{\mP}_h^k$ in \eqref{eq:coarse:TG}, the subspaces in the coarse problems only contain two directions: the preconditioned gradient and the coarse grid correction from the coarser level.
\end{remark}

\subsection{Numerical tests}\label{sec:sub:NumericalTestsI}
We demonstrate the SESOP-MG-$m$ algorithm performance for two nonlinear problems and compare it with steepest descent (SD), Nesterov acceleration \cite{nesterov2018lectures}, and limited-memory BFGS (LBFGS) \cite{nocedal2006numerical}. The MG method proposed in \cite{wen2009line}---denoted ``MG-Line''---is adapted for this comparison, to illustrate the effectiveness of using history. We use SD as the relaxation for ``MG-Line''. For SESOP-MG-$m$, we employ Newton's method for the subspace minimization at each iteration. All the tests are performed on a laptop with 2.3GHz Intel Core i$9$.
 
\subsubsection*{Example I}
Consider the following variational problem on a uniform grid \cite{wen2009line},
\e
\left\{ 
\begin{array}{l}
\min_{u(x,y)} F\left(u(x,y)\right) \equiv \int_\Omega \frac{1}{2}\left|\nabla u(x,y)\right|^2+\gamma \left(u(x,y)e^{u(x,y)}-e^{u(x,y)}\right), \\
~~~~~~~~~~~~~~~~~~~~~~~~~~~~-f(x,y)u(x,y) dx dy \\
~~~~\text{such that}~~~~ u(x,y)=0 ~~~\text{on}~~~~ \partial \Omega,
\end{array}
\right.
\label{eq:nonlinearIVaria}
\ee
where $\nabla$ is the gradient, $\gamma=10$, $\Omega = [0,1]\times [0,1]$, and
$$
f(x,y) = \left( \left(9\pi^2+\gamma e^{ (x^2-x^3)\sin(3\pi y)}\right)\left(x^2-x^3\right)+6x-2\right)\sin\left(3\pi y\right).
$$
Through the Euler-Lagrange equation, the PDE formulation of \eqref{eq:nonlinearIVaria} reads 
\e
\left\{
\begin{array}{c}
-\Delta u(x,y) + \gamma u(x,y) e^{u(x,y)}=f(x,y) ~~~~~~~~~\text{on}~~ \Omega, \\
~~~~~~~~~~~~u(x,y) = 0~~~~~~~~~~~~~~~~~~~~~~~~~~~~~~~~~\text{on}~~\partial \Omega,
\end{array}
\right.
\label{eq:NonlinearI:PDE}
\ee 
where $\Delta$ denotes the Laplacian.

Equation \eqref{eq:nonlinearIVaria} is discretized by finite differences, using first order forward differences for $\nabla$ in \eqref{eq:nonlinearIVaria}, resulting in a five-point stencil for $\Delta$ in \eqref{eq:NonlinearI:PDE}. The finest grid size is $1024\times 1024$, and we employ $7$ levels, coming down to grids of size $8\times 8$.  The coarse problems are defined by rediscretization, and full-weighted residual transfers and bilinear interpolation are used as the restriction and prolongation. The relaxation sweep parameters are $\nu_1=1$ and $\nu_2=0$ for MG-Line. BFGS with up to ten iterations is used for solving the problem on the coarsest level. The minFunc toolbox \cite{schmidt2005minfunc} is used for BFGS and Newton's method. The maximal number of iterations of SESOP-MG-$1$ and MG-Line are set to be $30$ and $100$, respectively. The other methods are allowed to use $500$ iterations. For clarity of display, we denote by $F^*+10^{-8}$ the minimal objective value over all methods after running the maximal allowed number of iterations. 

From \Cref{fig:NonProI}, we clearly observe that SESOP-MG-$1$ is the fastest algorithm in terms of both iteration count and CPU time. Moreover, we note that SD, Nesterov, and LBFGS converge fast initially, but then slow down. This well-known phenomenon is a result of the fact that these methods cannot efficiently eliminate low-frequency error, unlike MG methods which use CGC. Note also that SESOP-MG-$1$ is significantly faster than MG-Line, demonstrating the effectiveness of introducing history for acceleration.

\begin{figure}[!htb]
\centering
\subfigure[]{\centering\includegraphics[width=0.46\textwidth]{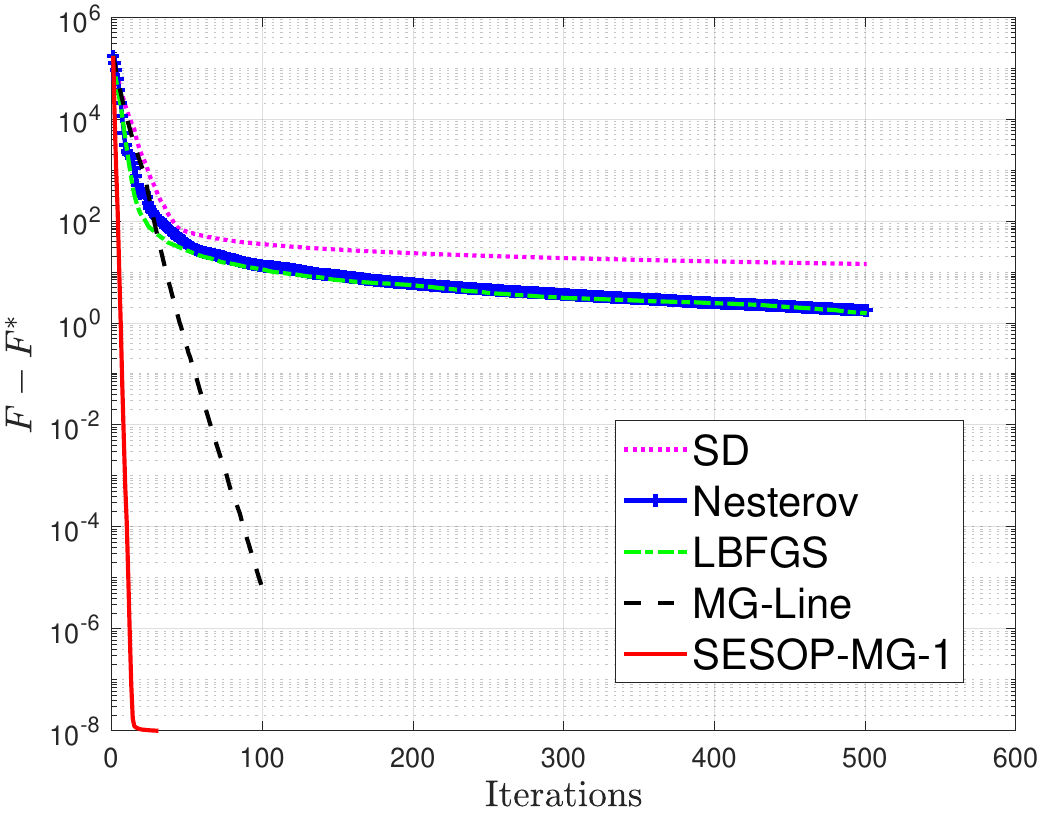}}
\centering
\subfigure[]{\centering\includegraphics[width=0.46\textwidth]{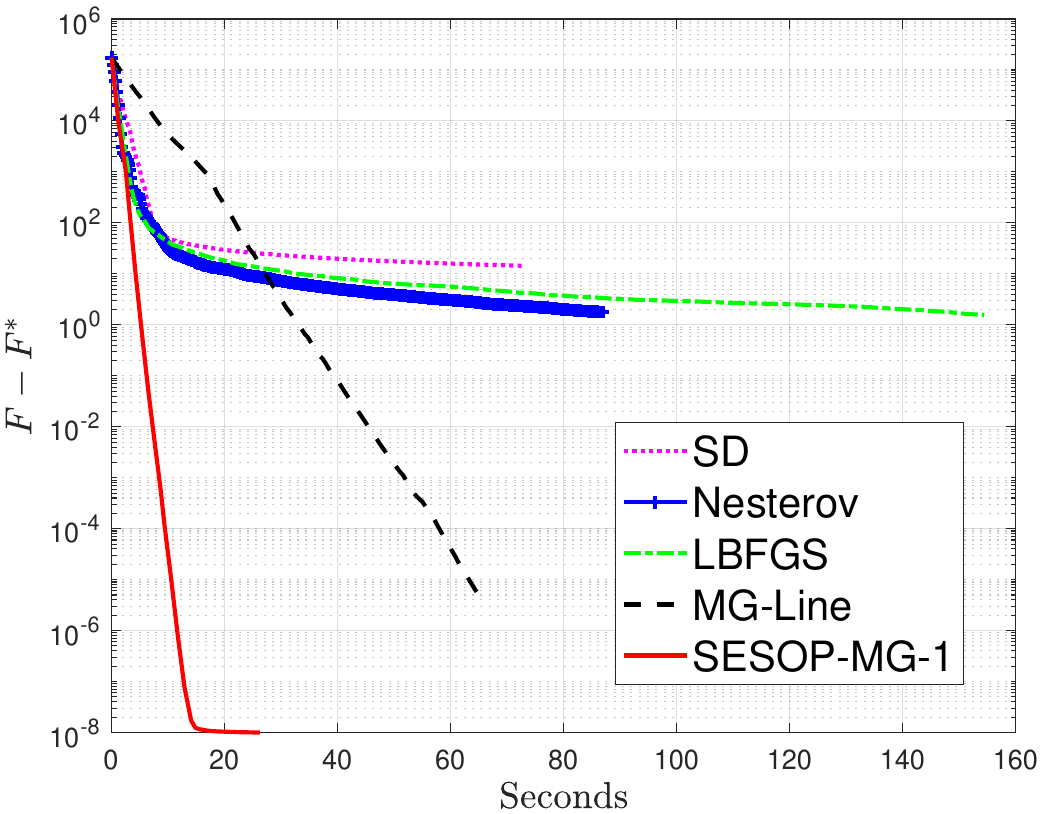}} 
\caption{Comparison of different methods for \eqref{eq:nonlinearIVaria} on $1024\times 1024$ grids.}
\label{fig:NonProI}
\end{figure}

\subsubsection*{Example II}
Our second example is the $p$-Laplacian,
\e
\left\{ 
\begin{array}{l}
\min_{u(x,y)} F\left(u(x,y)\right) \equiv \int_\Omega \left |\nabla u(x,y)+\xi\right |^p-f(x,y)u(x,y) dx dy, \\
~~~~\text{such that}~~~~ u(x,y)=0 ~~~\text{on}~~~~ \partial \Omega,
\end{array}
\right.
\label{eq:nonlinearIIvaria}
\ee
where $p\in(1,2)$. The corresponding PDE of \eqref{eq:nonlinearIIvaria} is:
\e
\left\{
\begin{array}{c}
-\nabla \cdot \left(\left | \nabla u+\xi \right |^{p-2} \nabla u\right)=f ~~~~~~~~~~\text{on}~~ \Omega \\
~~~~~~~~~u(x,y) = 0~~~~~~~~~~~~~~~~~~~~~~~~~~\text{on}~~\partial \Omega.
\end{array}
\right.
\label{eq:nonlinearIIvariaPDE}
\ee
The parameter $\xi=10^{-6}$ is introduced to maintain differentiability (hence a positive denominator). The function $f(x,y)$ is defined by substituting $u(x,y)=(x^2-x^3)\sin(3\pi y)$ into \eqref{eq:nonlinearIIvariaPDE}. Note that solving \eqref{eq:nonlinearIIvaria} becomes especially challenging when $p$ is close to $1$. We choose $p=1.3~\text{and}~1.6$ to study the performance of our approach. The experimental setting is the same as that of \emph{Example I}, except the maximal number of iterations allowed. As seen in Figure \Cref{fig:NonProII}, SESOP-MG-$1$ is the fastest of the five methods, followed by MG-Line, as in \emph{Example I}.


\begin{figure}[!htb]
\centering
\subfigure[$p=1.3$.]{\centering\includegraphics[width=0.46\textwidth]{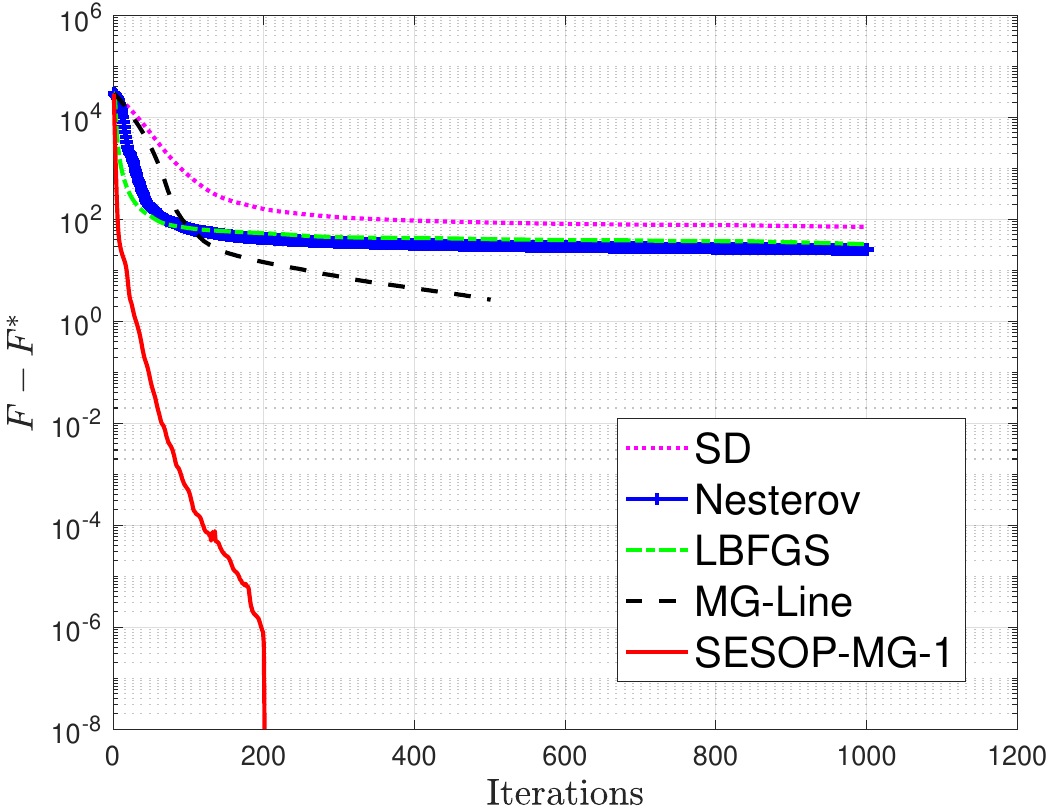}\label{fig:NonProII:iterp1.3}}
\centering
\subfigure[$p=1.3$.]{\centering\includegraphics[width=0.46\textwidth]{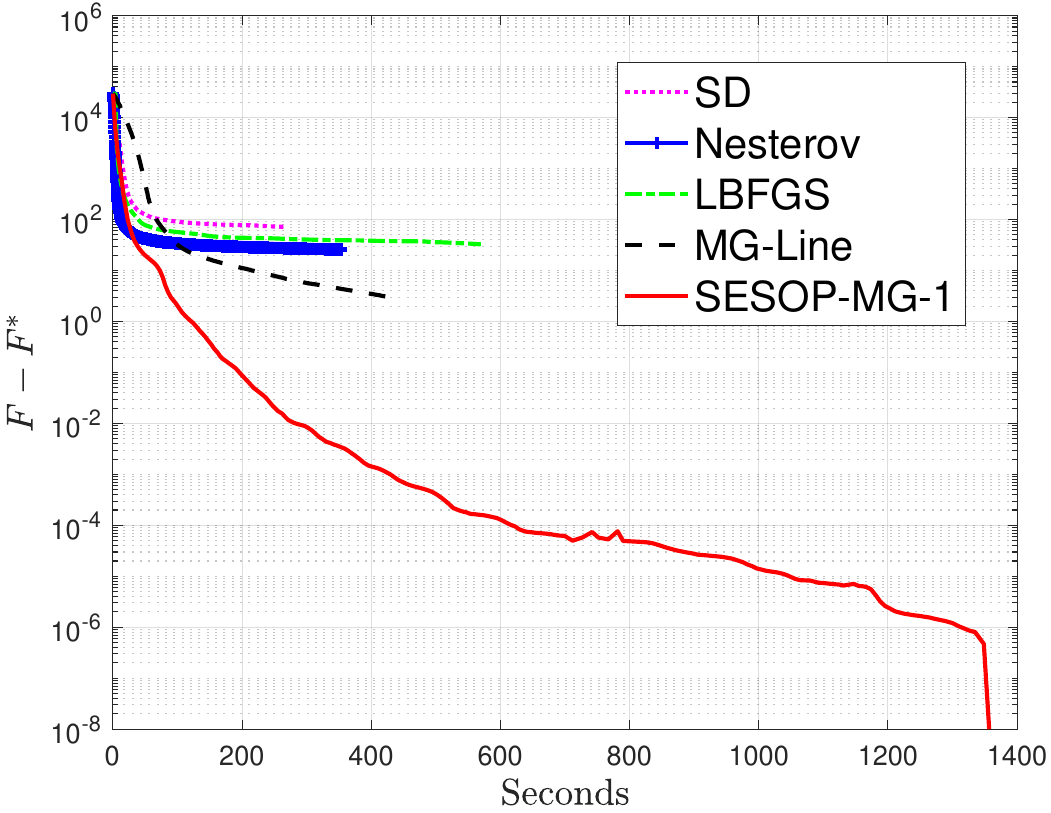}\label{fig:NonProII:CPUp1.3}}

\subfigure[$p=1.6$.]{\centering\includegraphics[width=0.46\textwidth]{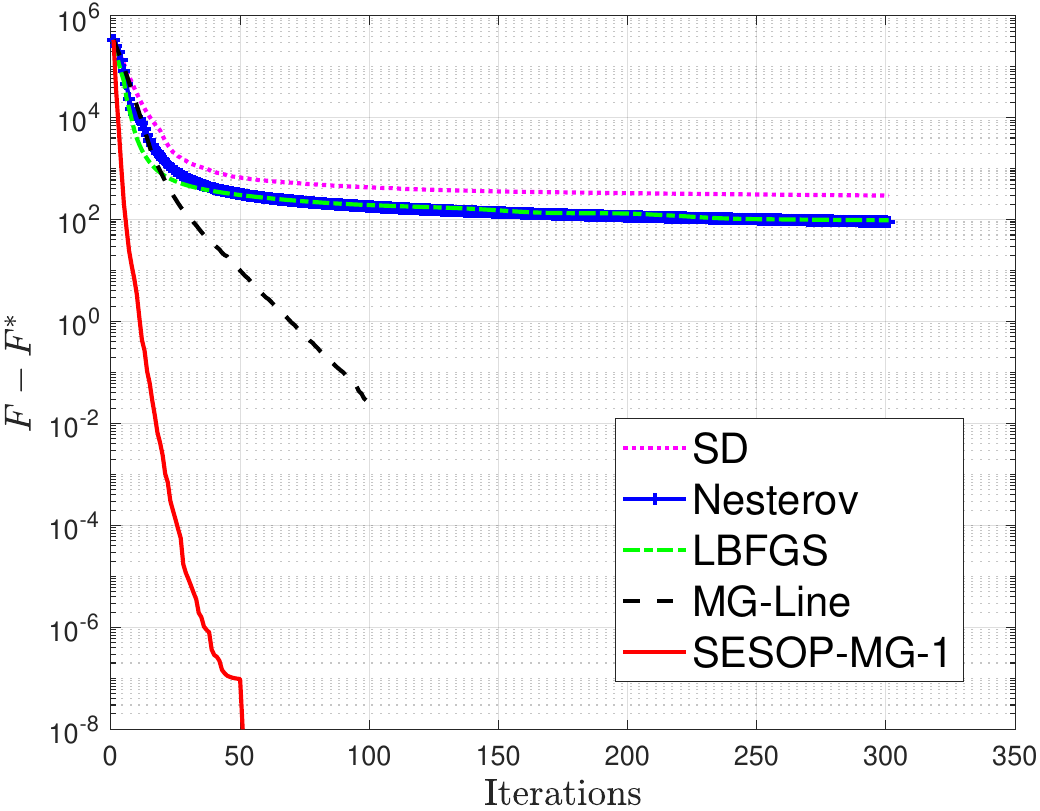}\label{fig:NonProII:iterp1.6}}
\centering
\subfigure[$p=1.6$.]{\centering\includegraphics[width=0.46\textwidth]{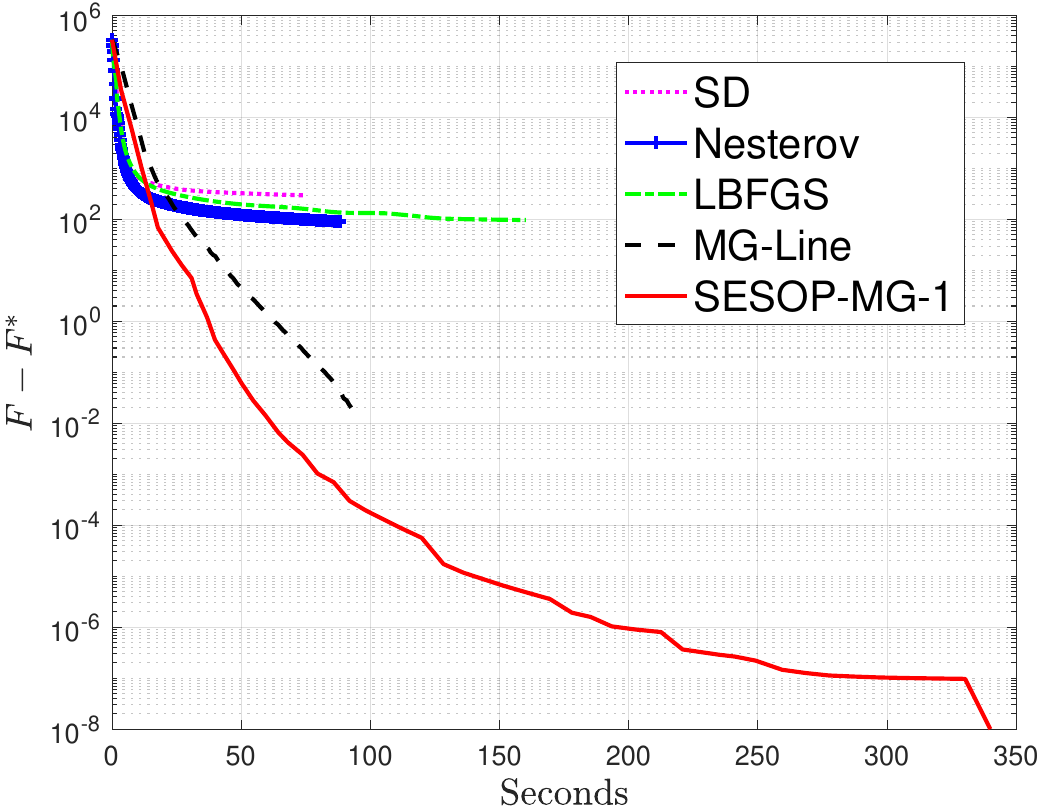}\label{fig:NonProII:CPUp1.6}}
\caption{Comparison of different methods for \eqref{eq:nonlinearIIvaria} on $1024\times 1024$ grids.}
\label{fig:NonProII}
\end{figure}

\subsubsection*{Dependence of SESOP-MG-$m$ on $m$} Next we show how the choice of the search-space size $m$ affects the performance of SESOP-MG-$m$. To this end,  we apply SESOP-MG-$m$ with different $m$ to \eqref{eq:nonlinearIIvaria} with $p=1.3$. In \Cref{fig:ComparhistSESOPMP:iter}, we see that using a larger $m$ yields faster convergence rates. However, a larger $m$ also increases the complexity of the subspace minimization results in higher CPU times per iteration. From \Cref{fig:ComparhistSESOPMP:CPU}, we find that $m=3$ is a good compromise for achieving low CPU times for this test. We also see that $m=5$ results in higher times than $m=3,4$ because, in practice, the larger size of the subspace introduces some numerical difficulties.

\begin{figure}[!htb]
\centering
\subfigure[Objective value versus iterations. $p=1.3$.]{\centering\includegraphics[width=0.46\textwidth]{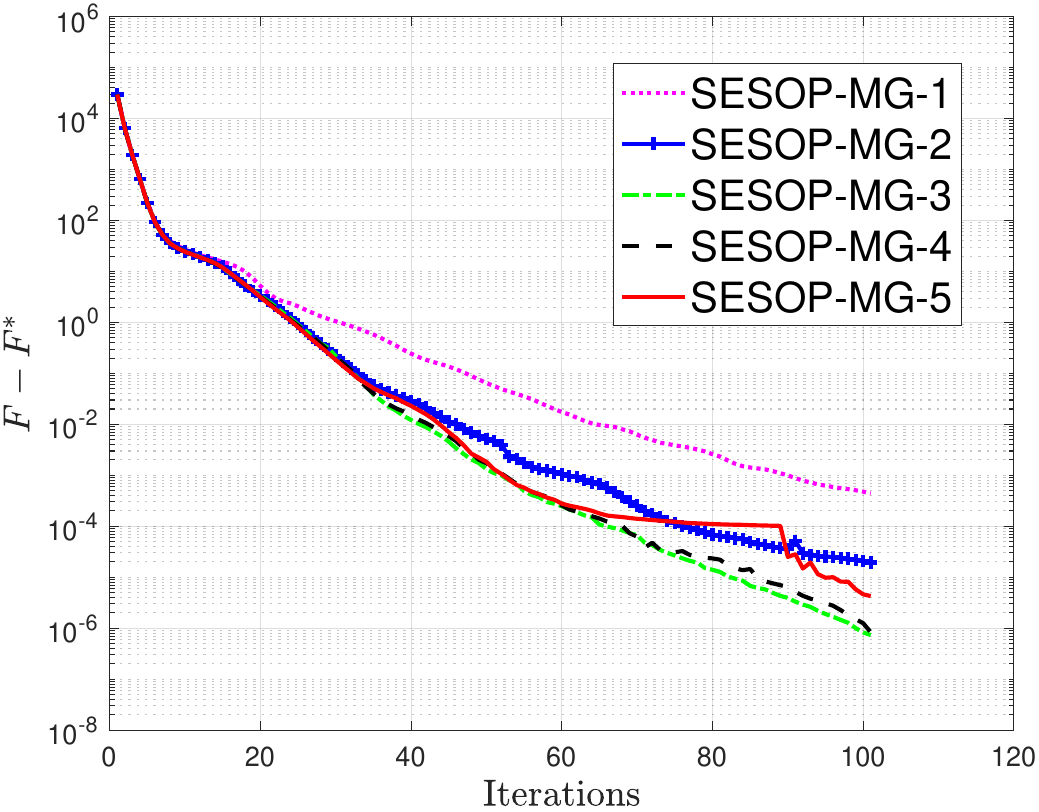}\label{fig:ComparhistSESOPMP:iter}}
\centering
\subfigure[Objective value versus CPU time. $p=1.3$.]{\centering\includegraphics[width=0.46\textwidth]{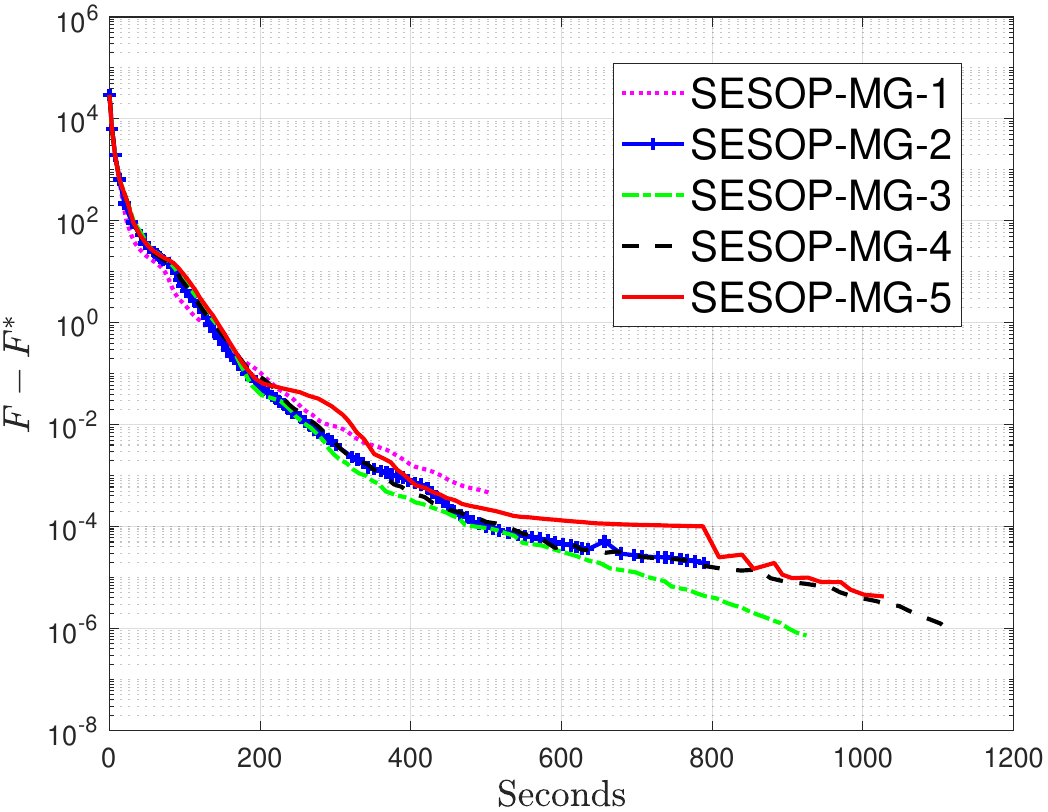}\label{fig:ComparhistSESOPMP:CPU}}

\caption{Comparison of different $m$ for \eqref{eq:nonlinearIIvaria} on $1024\times 1024$ grids.}
\label{fig:ComparhistSESOPMP}
\end{figure}

The numerical tests demonstrate the potential advantage of SESOP-TG/MG-$m$ for such types of optimization problems. However, SESOP-TG/MG-$m$ comes with the cost of solving a subspace minimization problem at each iteration. From \Cref{fig:ComparhistSESOPMP}, we observe that the performance of SESOP-TG/MG-$m$ deteriorates for large $m$. In the case of quadratic optimization problems, we may be able to avoid the subspace minimization, by using fixed nearly-optimal stepsizes for SESOP-TG-$1$. Indeed, we derive such fixed stepsizes, and show that they yield ACF's that are comparable to those obtained by subspace minimization. In such cases, we get the acceleration nearly for free, provided that we can estimate the fixed parameters efficiently. To this end, we propose two heuristic methods, based on local Fourier analysis (LFA) and smoothing analysis, to estimate the optimal fixed stepsizes cheaply.

\section{Convergence factor analysis of SESOP-TG-\texorpdfstring{$1$}{TEXT} for quadratic problems}\label{sec:ConvFactLinear}
To gain insight, we analyze SESOP-TG-$1$ for quadratic optimization problems, which are equivalent to the solution of linear systems. We first derive a fairly general formulation for SESOP-TG-$1$. Then we explore, under certain simplifying assumptions, a fixed-stepsize variant of SESOP-TG-$1$. In this analysis we assume for simplicity no pre- or post-relaxation steps. 

Consider the linear system 
\e
\mA\vx = \vf, \label{eq:2Dgeneral:BD:Dist:linear}
\ee
where $\mA\in\mathbb R^{N\times N}$ is a symmetric positive-definite (SPD) matrix, and we omit $h$ superscripts for notational simplicity. Evidently, solving \eqref{eq:2Dgeneral:BD:Dist:linear} is equivalent to the following quadratic minimization problem:
\e
\vx_*=\argmin_{\vx\in \mathbb R^N}  F^h(\vx) \triangleqtemp \frac{1}{2}\vx^\mathcal T \mA \vx - \vf^\mathcal T\vx. \label{eq:conv:analysis:QP}
\ee



\noindent Given iterates $\vx_{k-1}$ and $\vx_{k-2}$, the next iterate produced by SESOP-TG-$1$ is given by 
\e
\vx_k = \vx_{k-1}+c_1(\vx_{k-1}-\vx_{k-2})+c_2\mathcal{\mP}(\vf-\mA\vx_{k-1})+c_3\mI_H^h\mA_H^{-1}\mI_h^H(\vf-\mA\vx_{k-1}),\label{eq:update:history_1}
\ee
where $c_1,c_2,c_3$ are the optimal weights associated with the three directions comprising $\bar\mP_k^h$: $c_1$ multiplies the so-called history, that is, the difference between the last two iterates; $c_2$ multiplies the preconditioned gradient; $c_3$ multiplies the CGC direction $\vd^h_{k-1}$. Here, $\mA_H$ represents the coarse-grid matrix approximating $\mA$, which is most commonly defined by the Galerkin formula, $\mA_H = \mI_h^H \mA \mI_H^h$, or simply by rediscretization on the coarse-grid in the case where $\mA$ is the discretization of an elliptic PDE on the fine-grid.  
Subtracting $\vx_*$ from both sides of \eqref{eq:update:history_1}, and denoting the error by $\ve_k = \vx_* - \vx_{k}$, we get
\e
\ve_k = \ve_{k-1}+c_1(\ve_{k-1}-\ve_{k-2})-c_2 \mathcal{\mP} \mA\ve_{k-1}-c_3\mI_H^h\mA_H^{-1}\mI_h^H\mA\ve_{k-1}.
\label{eq:2d:defective}
\ee
Rearranging \eqref{eq:2d:defective} yields
\e
\ve_k = 
\bm B\ve_{k-1}-c_1\ve_{k-2},
\label{eq:2d:defective_rearrange}
\ee
where 
\e 
\bm B=(1+c_1)\mI-(c_2\mathcal{\mP}+c_3I_H^h\mA_H^{-1}I_h^H)\mA, 
\ee
and $\bm I$ denotes the identity matrix. Define the vector $\mE_k = \bmat \ve_k\\ \ve_{k-1} \emat$. Then, \eqref{eq:2d:defective_rearrange} implies the following relation:
\e
\mE_k = \bm\Upsilon\mE_{k-1}, ~~ \bm\Upsilon \triangleqtemp \bmat \bm B&-c_1\mI\\ \mI&\bm0\emat,
\label{eq:IterativeMatrixUpsilon}
\ee
and the asymptotic convergence factor (ACF) of SESOP-TG-$1$ is given by the spectral radius of $\bm \Upsilon$.

To analyze the ACF of SESOP-TG-$1$, we continue under the assumption that the coefficients $c_j$, $j=1,2,3$, are fixed, and compute the optimal coefficients, yielding the smallest ACF. Denote by $r$ an eigenvalue of $\bm\Upsilon$ with eigenvector $\vv = [\vv_1, \vv_2]^{\mathcal T}$,
$$ \bmat \bm B&-c_1\mI\\ \mI&\bm0\emat \bmat \vv_1 \\ \vv_2 \emat = r \bmat \vv_1 \\ \vv_2 \emat \, .
$$
Hence, $\vv_1 = r \vv_2$ and $\bm B \vv_1 - c_1 \vv_2 = r \vv_1$. This yields
$$
\bm B \vv_1 = \left(r + \frac{c_1}{r} \right) \vv_1.
$$
Thus, $\vv_1$ is an eigenvector of $\bm B$ with eigenvalue $b \triangleqtemp r + \frac{c_1}{r}$ leading to
$$
r^2 -b r + c_1 = 0,
$$
with solutions
$$
r_1(b,c_1) = \frac{1}{2} \left( b + \sqrt{b^2 - 4c_1} \right)~\text{and}~r_2(b,c_1)=\frac{1}{2} \left( b-\sqrt{b^2 - 4c_1} \right).
$$
The spectral radius of $\bm \Upsilon$ is therefore
$$
\rho(\bm\Upsilon) = \frac{1}{2}\max_b \left|b+\text{sgn}(b)\sqrt{b^2-4c_1}\right|,
$$
where $b$ runs over the eigenvalues of $\bm B$ and $\text{sgn}(\cdot)$ is a sign function which is defined to be 1 for non-negative arguments and -1 for negative arguments. For the remainder of our analysis we focus on the case where the eigenvalues $b$ of $\bm B$ are all real. 

\subsection[]{The case of real {\em b}} \label{subsection:real:b}
In many practical cases, the eigenvalues of $\bm B$ are all real, which simplifies the analysis and yields insight. We focus next on a common situation where this is indeed the case.

\begin{lemma} \label{lem:real:eigenvalues:Gamma}
Assume:
\begin{enumerate}
\item
The prolongation has full column-rank and the restriction is the adjoint of the prolongation: $I_h^H = (I_H^h)^{\mathcal T}$.
\item
The coarse-grid operator $\mA_H$ is SPD.
\item
The preconditioner $\mathcal{\mP}$ is SPD.
\end{enumerate}
Then the eignevalues of $ \bm B=(1+c_1)\mI-(c_2\mathcal{\mP}+c_3\mI_H^h\mA_H^{-1}\mI_h^H)\mA $ are all real. 
\end{lemma}
\begin{proof}
Because $\mA$ is SPD, there exists a SPD matrix $\mA^{\frac{1}{2}}$ such that $\mA^{\frac{1}{2}} \mA^{\frac{1}{2}} = \mA$. The matrix 
$$
\mA^{\frac{1}{2}} \bm B {\mA^{-\frac{1}{2}}} = (1+c_1)\mI-\mA^{\frac{1}{2}}(c_2\mathcal{\mP}+c_3\mI_H^h\mA_H^{-1}\mI_h^H)\mA^{\frac{1}{2}}
$$
is similar to $\bm B$ so they have the same eigenvalues. Moreover, $\mA^{\frac{1}{2}}\bm B {\mA^{-\frac{1}{2}}}$ is evidently symmetric, so its eigenvalues are all real.
\end{proof}
The first two assumptions of this lemma are satisfied commonly, including of course both the case where $\mA_H$ is defined by rediscretization on the coarse grid and the case of Galerkin coarsening. The preconditioner $\mathcal{\mP}$ is SPD for some commonly used MG relaxation methods, including Richardson (where $\mathcal{\mP}$ is the identity matrix), Jacobi (where $\mathcal{\mP}$ is the inverse of the diagonal of $\mA$), and symmetric Gauss-Seidel.

We next adopt the change of variables $c_{23} = c_2 + c_3$, $\alpha = c_2 / c_{23}$. This yields 
\e 
\bm B=(1+c_1)\mI-c_{23}\mA_\alpha, 
\ee
with
\e
\mA_\alpha = \left(\alpha\mathcal{\mP}+ (1-\alpha) \mI_H^h\mA_H^{-1}\mI_h^H\right)\mA.\label{eq:A_Alpha}
\ee

\begin{lemma} \label{lem:P:Posdef}
The eigenvalues of $\mA_\alpha$ are real and positive for any $\alpha \in (0,1]$.
\end{lemma}
\begin{proof}
This follows from the fact that $\mA$, $\mathcal{\mP}$ and $I_H^h\mA_H^{-1}I_h^H$ are SPD. 
\end{proof}
We henceforth denote the eigenvalues of $\mA_\alpha$ by $a_{\alpha}$, and assume $\alpha \in (0,1]$ (that is, $c_2$ and $c_3$ are of the same sign), so the $a_{\alpha}$'s are all real and positive. We then proceed by fixing $\alpha$ and optimizing $c_1$ and $c_{23}$ so as to minimize the spectral radius of $\mUpsilon$. In light of \Cref{lem:real:eigenvalues:Gamma}, we have $b\equiv b(c_1,c_{23},a_{\alpha}) = 1+c_1-c_{23}a_{\alpha}$, and the ACF as a function of $c_1,~c_{23}$ is given by:
\e
\bar r(c_1,c_{23}) \triangleqtemp \rho(\bm\Upsilon)=\max_{a_{\alpha}} \hat r(c_1,c_{23},a_\alpha),\label{eq:worst_conv_rate:def}
\ee
where $\hat r(c_1,c_{23},a_\alpha)\triangleqtemp\frac{1}{2}\left|b+\text{sgn}(b)\sqrt{b^2-4c_1}\right|$. For convenience, we use $\bar r_{c_1}(c_{23})$ (respectively, $\bar r_{c_{23}}(c_1)$) to denote $\bar r(c_1,c_{23})$ considered as a single-variable function with a fixed $c_1$ (respectively, $c_{23}$).
Our parameter optimization problem is now defined by
$$
(c_1^*, c_{23}^*) = \argmin_{(c_1,c_{23})} \bar r(c_1,c_{23}).
$$
The following three lemmas provide us with a closed-form solution of this minimization problem, if the smallest and largest eigenvalues of $\mA_\alpha$ are given.
\begin{lemma}\label{lemma_1:lar:sma:ACF}
For any fixed $c_1$ and $c_{23}$, $\hat r(c_1,c_{23},a_\alpha)$ is maximized at either $a_{\alpha}^{min}$ or $a_{\alpha}^{max}$, the smallest and largest eigenvalues of $\mA_\alpha$, respectively.
\end{lemma}
\begin{proof}
See \Cref{proof:lemma_1:lar:sma:ACF}.
\end{proof}

\begin{lemma}\label{lemma_2:determ:c_23}
For any fixed $c_1$, $c_{23}^* = 2(1+c_1) / (a_{\alpha}^{max} + a_{\alpha}^{min})$.
\end{lemma}
\begin{proof}
See \Cref{proof:lemma_2:determ:c_23}.
\end{proof}
Plugging $c_{23}^*$ into \eqref{eq:worst_conv_rate:def} yields
\e
\bar r_{c_{23}^*}(c_1) = \frac{1}{2}\left|\mu (1+c_1)+\sqrt{\mu^2(1+c_1)^2-4c_1}\right|,
\label{eq:worst_conv_rate:def:abs:no_c_2}
\ee
where $\mu=(\kappa-1)/(\kappa+1)\in (0,1)$ with $\kappa=\frac{a_{\alpha}^{max}}{a_{\alpha}^{min}}>1$ the condition number of $\mA_\alpha$. 
\Cref{lemma_3:determ:c_1} provides the optimal $c_1$, which minimizes $\bar r_{c_{23}^*}(c_1)$ in \eqref{eq:worst_conv_rate:def:abs:no_c_2}. 
\begin{lemma}\label{lemma_3:determ:c_1}
$c_1^*=\left( \frac{\sqrt{\kappa}-1}{\sqrt{\kappa}+1} \right)^2$.
\end{lemma}
\begin{proof}
See \Cref{proof:lemma_3:determ:c_1}.
\end{proof}
Substituting the expression of  $c_1^*$ into \eqref{eq:worst_conv_rate:def:abs:no_c_2}, we get the worst-case convergence factor $r^* $ with optimal $c_1^*$ and $c_{23}^*$
\e
\bar r^* \triangleqtemp \bar r(c_1^*,c_{23}^*)=\frac{\sqrt{\kappa}-1}{\sqrt{\kappa}+1}. \label{eq:OPtimalCovergenceHistory}
\ee
Note that $c_1^*=(\bar r^*)^2$ indicates that the history term becomes significant if the problem is ill-conditioned.
\begin{remark}
\label{remark:AnalysisSummary}

The optimal ACF of the fixed-coefficient variant of SESOP-TG-$1$ is found to be $\bar r^* = \frac{\sqrt{\kappa}-1}{\sqrt{\kappa}+1}$, where $\kappa$ is the condition number of $\mA_\alpha$, optimized over $\alpha \in (0,1]$. (The optimal choice of $\alpha$ will be discussed further below.) By using the definition of $\kappa$, the optimal coefficient for the preconditioned gradient term as specified in \Cref{lemma_2:determ:c_23} can be written as
\e
c_{23}^* = \frac{4}{a_{\alpha}^{min}(\sqrt{\kappa}+1)^2}.
\label{eq:SimplifiedOptimalC23}
\ee
By setting $c_1 = 0$, we retain SESOP-TG-$0$ and then the optimal $c_{23}$ becomes
\e
c_{23}^* = \frac{2}{a_{\alpha}^{min}(\kappa+1)}.
\label{eq:OptimalC23NoHistory}
\ee
Evidently, the asymptotic convergence factor of  SESOP-TG-$0$ is
\e 
\bar r^* = \frac{\kappa - 1}{\kappa + 1}.
\label{eq:OptimalConvergenceFactorNoHistory}
\ee
Comparing \eqref{eq:OPtimalCovergenceHistory} with \eqref{eq:OptimalConvergenceFactorNoHistory}, we see the significant improvement provided by the use of a single history direction, with the condition number replaced by the square root. This result is reminiscent of the conjugate gradients (CG) method compared to steepest descent (SD) for quadratic problems, and indeed there is an equivalence between SESOP and CG (respectively, SD) for the case of single history (respectively, no history) with $\alpha=1$. We note that the condition number in our scheme is that of $\mA_\alpha$, not $\mA$, which differs from the case where we do not use the CGC direction. The advantage of SESOP is in allowing the addition of various search directions, at the cost of optimizing the coefficients. This study is partly aimed at reducing this cost.
\end{remark}

\subsection{Towards optimizing the condition number of \texorpdfstring{${\mA}_\alpha$}{TEXT}}\label{sec:subInsightsofOptCond}
The previous analysis indicates that we should aim to minimize $\kappa$, the condition number of $\mA_\alpha$. In certain cases, particularly when $\mA$ is a circulant matrix (typically the discretization of an elliptic PDE with constant coefficients on a rectangular or infinite domain), this can be done by means of Fourier analysis \cite{wienands2004practical}, as discussed in \Cref{sec:sub:optCondPractice}. Here, we begin with a more general discussion to gain insight into this matter. We consider the case of Galerkin coarsening, $\mA_H = (\mI_H^h)^\mathcal T \mA \mI_H^h$, and $\mathcal{\mP} = \mI$. Guided by  
\cite{falgout2005two}, we consider the case where the columns of the prolongation matrix $\mI_H^h$ are comprised of a subset of the eigenvectors of $\mA$.  

\begin{lemma} \label{lem:ConditionNumber}
Let $\{a_i,\va_i\}$, $i = 1, \ldots , N,$ denote the eigenvalues and eigenvectors of $\mA$. Assume that the columns of the prolongation matrix $\mI_H^h$ are comprised of a subset of the eigenectors of $\mA$, and denote by $\mathcal R (\mI_H^h)$ the range of the prolongation, that is, the subspace spanned by the columns of $\mI_H^h$. Denote
\begin{eqnarray*}
a_{fmax}  =  \max_{i:\va_i \notin \mathcal R (\mI_H^h)}  a_i,~~~~~&
a_{fmin}  =  \min_{i:\va_i \notin \mathcal R (\mI_H^h)}  a_i, \\
a_{cmax}  =  \max_{i:\va_i \in \mathcal R (\mI_H^h)}  a_i, ~~~~~&
a_{cmin}  =  \min_{i:\va_i \in \mathcal R (\mI_H^h)}  a_i.
\end{eqnarray*}
Then, the condition number of $\mA_\alpha$ in \eqref{eq:A_Alpha} with $\mathcal{\mP} = \mI$ is given by 
\e \label{eq:Kappa}
\kappa = \frac{\max(\alpha a_{fmax}, ~\alpha a_{cmax}+1-\alpha)}{\min(\alpha a_{fmin}, ~\alpha a_{cmin}+1-\alpha)}.
\ee
\end{lemma}
\begin{proof}
See \Cref{proof:lemma_3:idealProRest}.
\end{proof}
We see that the second term in $\mA_\alpha$, which corresponds to the direction provided by the CGC, increases the eigenvalues associated with the columns of the prolongation by $1-\alpha$. It thus follows from \eqref{eq:Kappa} that, to obtain any advantage at all from the CGC direction in reducing $\kappa$, the eigenvector associated with the smallest $a_i$ must be included amongst the columns of the prolongation, and therefore $ a_{cmin} \leq  a_{fmin}$. Similarly, considering the numerator in \eqref{eq:Kappa}, it is clearly advantageous that the eigenvector corresponding to the largest $ a_i$ not be included in the range of the prolongation. With these assumptions, we obtain the following result for the optimal $\alpha$ which minimizes $\kappa$.
\begin{theorem} \label{theorem:OptimalAlpha}
Assume $ a_{cmin} \leq  a_{fmin}$ and $ a_{cmax} \leq  a_{fmax}$. Then, the condition number $\kappa$ is minimized by choosing
$$
 \alpha_{opt} \triangleqtemp  \frac{1}{1+ a_{fmin}- a_{cmin}} \leq 1,
$$
resulting in the optimal $\kappa$,
\e \label{eq:KappaOpt}
\kappa_{opt} \triangleqtemp \left\{ \begin{array}{ll} \frac{a_{fmax}}{a_{fmin}} & {\rm if}~ a_{fmax}-a_{fmin} \geq  a_{cmax} -  a_{cmin}, \\ 1 + \frac{a_{cmax}-a_{cmin}}{a_{fmin}} & {\rm otherwise.} \end{array}\right.
\ee
\end{theorem} 
\begin{proof}
See \Cref{proof:theoremOptalpha:idealProRest}.
\end{proof}
\begin{remark}
\label{remark:kappaOptTheoretic}
Notice that $\kappa_{opt}$ is either equal to $a_{fmax} /a_{fmin}$ or at most ``slightly larger'' because 
$$
1 + \frac{a_{cmax}-a_{cmin}}{a_{fmin}} = \frac{a_{fmax}}{a_{fmin}} + 1 - \frac{a_{fmax}-a_{cmax}+a_{cmin}}{a_{fmin}} < \frac{a_{fmax}}{a_{fmin}} + 1 \, .
$$
That is, even in the regime $a_{fmax} - a_{fmin} < a_{cmax} - a_{cmin}$, the optimal condition number $\kappa_{opt}$ is increased by less than 1. Note that $\kappa =  a_{fmax} /  a_{fmin}$ yields a convergence factor (with no history) of $\mu = \frac{\kappa-1}{\kappa+1}$, which matches that of the classical TG algorithm with optimally weighted Richardson relaxation followed by CGC.  In \Cref{sec:connection_h_Ellipticity}, we will show the connection of $\kappa_{opt}$ presented here with the so-called $h$-ellipticity measure, which is a qualitative criterion for the existence of local smoothers for a given elliptic PDE \cite{trottenberg2000multigrid}. Finally, we note that the optimal prolongation is obtained by choosing the columns of $I_H^h$ to be the eigenvectors associated with the smallest eigenvalues of $\mA$ (similarly to the classical TG case \cite{falgout2005two}). This clearly minimizes the ratio $ a_{fmax} /  a_{fmin}$. Furthermore, this choice yields $ a_{cmax} \leq  a_{fmin}$, so if $ a_{fmax} - a_{fmin} <  a_{cmax} -  a_{cmin}$ (so $\kappa_{opt} >  a_{fmax} /  a_{fmin}$), then we obtain $\kappa_{opt} = 1 + \frac{a_{cmax} -  a_{cmin}}{a_{fmin}}< 2$.
\end{remark}

\subsection{Optimizing the condition number of \texorpdfstring{${\mA}_\alpha$}{TEXT} in practice}\label{sec:sub:optCondPractice}
As our analysis indicates, optimizing $\alpha$ so as to minimize $\kappa$ is a crucial step towards obtaining an optimal convergence factor. In general, the complexity of minimizing $\kappa$ is very high and the explicit results of \Cref{sec:subInsightsofOptCond} are only valid when the prolongation is comprised of eigenvectors of $\mA$, which is impractical, because the prolongation needs to be very sparse for efficiency. 
In certain cases, such as when $\mA$ results from discretizing elliptic PDEs with constant or slowly varying coefficients, Local Fourier Analysis (LFA) \cite{brandt1977multi,wienands2004practical} can be used in conjunction with the analysis of \Cref{sec:subInsightsofOptCond} to yield effective approximate results.

LFA, also called Local Mode Analysis, is a useful quantitative tool for estimating the asymptotic convergence factor ACF of MG methods for elliptic PDEs with constant coefficients. We describe LFA here briefly, and for further details refer the reader to standard multigrid textbooks or to \cite{wienands2004practical}, a comprehensive book on LFA for MG. We consider the two-dimensional case for simplicity. For $\mA$ that is a discretization of a constant-coefficient elliptic PDE on an infinite or doubly periodic domain of uniform mesh-size $h$, grid-based functions of the form $\psi(\vtheta,x,y)=e^{\iota \theta_1 x/h}e^{\iota \theta_2 y/h}$ with $\vtheta=(\theta_1,\theta_2)\in [-\pi,\pi)^2$ and $\iota = \sqrt{-1}$, are eigenfunctions of $\mA$. Furthermore, if standard shift-invariant prolongation is used, such as bilinear or bicubic interpolation, then appropriate subsets of dimension four of the eigenfunctions $\psi(\vtheta,x,y)$ form subspaces that are invariant under multiplication by $\mA_\alpha$ (as in standard two-level analysis). The upshot is that we can compute the eigenvalues of $\mA_\alpha$ at a fairly moderate cost, and use linesearch to find $\alpha$ which optimizes the condition number of $\mA_\alpha$. This approach is demonstrated later in numerical examples, including the option of saving computational cost by optimizing $\alpha$ on relatively coarse grids.        

We can furthermore obtain a very cheap approximation to the optimal $\alpha$ and $\kappa$ by making the simplifying assumptions that are commonly used in computing the so-called smoothing factor of relaxation (known as smoothing analysis) as follows. Denote by {$a(\vtheta)$} the eigenvalue of $\mA$ associated with the grid-function $\psi(\vtheta,x,y)$, and partition the $\vtheta$ domain into low- and high-frequencies as in \Cref{definition:highlowcomp}. 

\begin{definition}[Low- and High-frequency components]\label{definition:highlowcomp}
\en
\begin{array}{rcl}
a(\vtheta) ~~ \text{is a low-frequency component}& \Longleftrightarrow & \vtheta\in T^{\text{low}}:=\left[-\frac{\pi}{2},\frac{\pi}{2}\right)^2,\\
a(\vtheta) ~~ \text{is a high-frequency component}& \Longleftrightarrow & \vtheta\in T^{\text{high}}:=\left[-\pi,\pi\right)^2 \setminus \left[-\frac{\pi}{2},\frac{\pi}{2}\right)^2.
\end{array}
\een
\end{definition} 

\noindent
Smoothing analysis simplifies by assuming that the CGC acts as a projection onto the high-frequency subspace, that is, it has no effect on high-frequency error components, while it eliminates exactly low-frequency error components. Under these simplifying assumptions, we obtain
\begin{equation}
\label{eq:eta:highlowFreqFourierModes}
\begin{array}{r@{}l@{}r@{}l}
 a_{fmax}  = &{} \max_{\vtheta\in T^{\text{high}}} a(\vtheta),~~~~&{}
 a_{fmin}  =  &{}\min_{\vtheta\in T^{\text{high}}}  
 a(\vtheta),\\
 a_{cmax}  = &{} \max_{\vtheta\in T^{\text{low}}}  a(\vtheta),~~~~&{}
 a_{cmin}  =  &{}\min_{\vtheta\in T^{\text{low}}} 
 a(\vtheta).
\end{array}
\end{equation}
Moreover, the eigenvalues of the second term in $\mA_\alpha$ are given by $\gamma(\vtheta) \triangleqtemp \frac{a(\vtheta)}{a_H(2\vtheta)}$ with $a_H(2\vtheta)$ the eigenvalues of $\mA_H$ for $\vtheta\in T^{\text{low}}$ \cite{yavneh1998coarse}. Now, we can write $\kappa$ as:
\e \label{eq:KappaSmoothAnalysis}
\kappa = \frac{\max\left(\alpha  a_{fmax}, ~\alpha  a_{cmax}+\left(1-\alpha\right)\gamma_{max}(\vtheta)\right)}{\min\left(\alpha  a_{fmin}, ~\alpha  a_{cmin}+\left(1-\alpha\right)\gamma_{min}(\vtheta)\right)},
\ee
where $\gamma_{max}(\vtheta)$ and $\gamma_{min}(\vtheta)$ denote the maximum and minimum value of $\gamma(\vtheta)$, respectively. Similarly to \Cref{theorem:OptimalAlpha}, the optimal $\alpha$ is given by 
\e
\label{eq:alphaoptSmoothAnaly}
\alpha_{opt}=\frac{1}{1+\frac{a_{fmin}- a_{cmin}}{\gamma_{min}(\vtheta)}}.
\ee
We note that, for elliptic PDEs, the accuracy of using the approximations \eqref{eq:eta:highlowFreqFourierModes} and $\gamma(\vtheta)$ depend on the order of intergrid transfers \cite{hemker1990order,yavneh1998coarse}. In the following numerical tests, we evaluate experimentally the accuracy of using \eqref{eq:alphaoptSmoothAnaly} to estimate the fixed stepsizes in practice.

\subsection{A connection with the $h$-ellipticity measure}\label{sec:connection_h_Ellipticity}
Using \Cref{definition:highlowcomp}, we define an ``idealized'' MG method as a two-grid method that affects high-frequency error components only on the fine grid and eliminates all low-frequency error components via the CGC, corresponding to \eqref{eq:eta:highlowFreqFourierModes}. Then, the ACF of an idealized SESOP-TG-$1$ is $\frac{\sqrt{\kappa_{opt}}-1}{\sqrt{\kappa_{opt}}+1}$ with $\kappa_{opt} = \frac{a_{fmax}}{a_{fmin}}$, corresponding to the discussion of \Cref{sec:subInsightsofOptCond} for ill-conditioned problems. Alternatively, we can express this idealized ACF in terms of $E_h(\mA) = \frac{1}{\kappa_{opt}}$, which is known in the literature as the $h$-ellipticity measure, obtaining $r=\frac{1-\sqrt{E_h}}{1+\sqrt{E_h}}$. For SESOP-TG-$0$, in contrast, the ACF becomes $r = \frac{1-E_h}{1+E_h}$, which is the well-known smoothing factor of optimally damped Jacobi relaxation for symmetric problems \cite{trottenberg2000multigrid}.    



\subsection{Numerical tests---continued}
\label{sec:sub:NumericalTestsContinued}
In this section, we first test the accuracy of using \eqref{eq:IterativeMatrixUpsilon} to predict the ACF of SESOP-TG-$1$ with fixed stepsizes. Then, compare the ACF with minimization over the subspace to that obtained with optimized fixed stepsizes. Finally, we examine the performance of the two proposed heuristic methods (cf. \Cref{sec:sub:optCondPractice}) for estimating the fixed stepsizes. The rotated anisotropic diffusion problem is chosen as the test problem,
\e
u_{ss}+\epsilon u_{tt} = f , \label{eq:anis_diff_con}
\ee
where $ u_{ss}$ and $u_{tt}$ denote the second partial derivatives of $u$ in the $(s,t)$ coordinate system. Denote by $\phi$ the angle between $(s,t)$ and the grid-aligned coordinate system, $(x,y)$. We rewrite \eqref{eq:anis_diff_con} as
\e
(C^2+\epsilon S^2) u_{xx}+2(1-\epsilon)CS  u_{xy}+(\epsilon C^2+S^2) u_{yy}=f, \label{eq:anis_diff_con:xy}
\ee
where $C=\cos\phi$ and $S=\sin\phi$. Using a nine-point stencil,
$$\bmat -\frac{1}{2}(1-\epsilon)CS&\epsilon C^2+S^2& \frac{1}{2} (1-\epsilon)CS\\ C^2+\epsilon S^2& -2(1+\epsilon) & C^2+\epsilon S^2\\ \frac{1}{2}(1-\epsilon)CS &\epsilon C^2+S^2& -\frac{1}{2}(1-\epsilon)CS \emat,
$$
to discretize \eqref{eq:anis_diff_con:xy} on a uniform grid with mesh-size $h$ and prescribed boundary conditions, we get a linear system of the form
\begin{equation}
\mA^h\vu^h=\vf^h.\label{eq:discretization:linear}
\end{equation}



\subsection*{ACF Prediction}
In our first test, we simply choose $c_3 = 1$, and $\kappa = \frac{1}{E_h}$, and then use \Cref{lemma_2:determ:c_23,lemma_3:determ:c_1} to compute $c_1$ and $c_2$. The resulting algorithm is referred to as ``SESOP-TG-$1$-Fixed''. We discretize \eqref{eq:anis_diff_con:xy} on a $256\times256$ grid, imposing Dirichlet boundary conditions, and we employ bilinear prolongation and full-weighted residual transfers. The coarse-grid operators are defined by direct rediscretization. Finally, we set $\mathcal P=\mI$, i.e., no preconditioning. The ACF achieved in practice by SESOP-TG-$1$-Fixed is evaluated as the geometric mean of the convergence factor per iteration in the last $5$ iterations, which are terminated at $500$ iterations or when the residual norm is smaller than $10^{-8}$, whichever comes first. The convergence factor at the $k$th iteration is defined by the ratio of the successive residual norms $\|\vr_k^h\|_2/\|\vr_{k-1}^h\|_2$, where $\vr_k^h=\vf^h-\mA^h\vu_k^h$ and $\vu_k^h$ is the approximate solution at the $k$th iteration. In \Cref{fig:thore:prac:rotated}, we see that the ACF predicted by the spectral radius of $\bm\Upsilon$ in \eqref{eq:IterativeMatrixUpsilon} matches the practical results well.
\begin{figure}[!htb]
	\centering
	\subfigure[$\phi=0$]{\centering \includegraphics[width=0.45\textwidth]{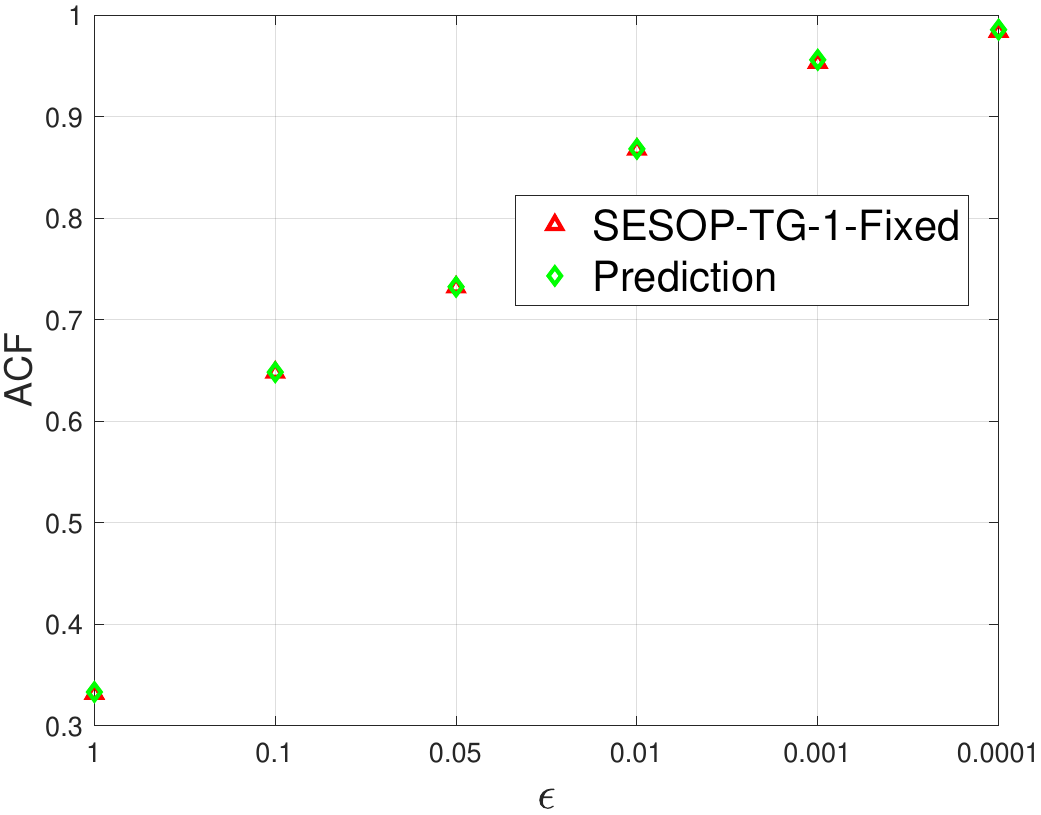}\label{fig:thore:prac:rotated:1}}
	\subfigure[$\phi=\frac{\pi}{100}$]{\centering \includegraphics[width=0.45\textwidth]{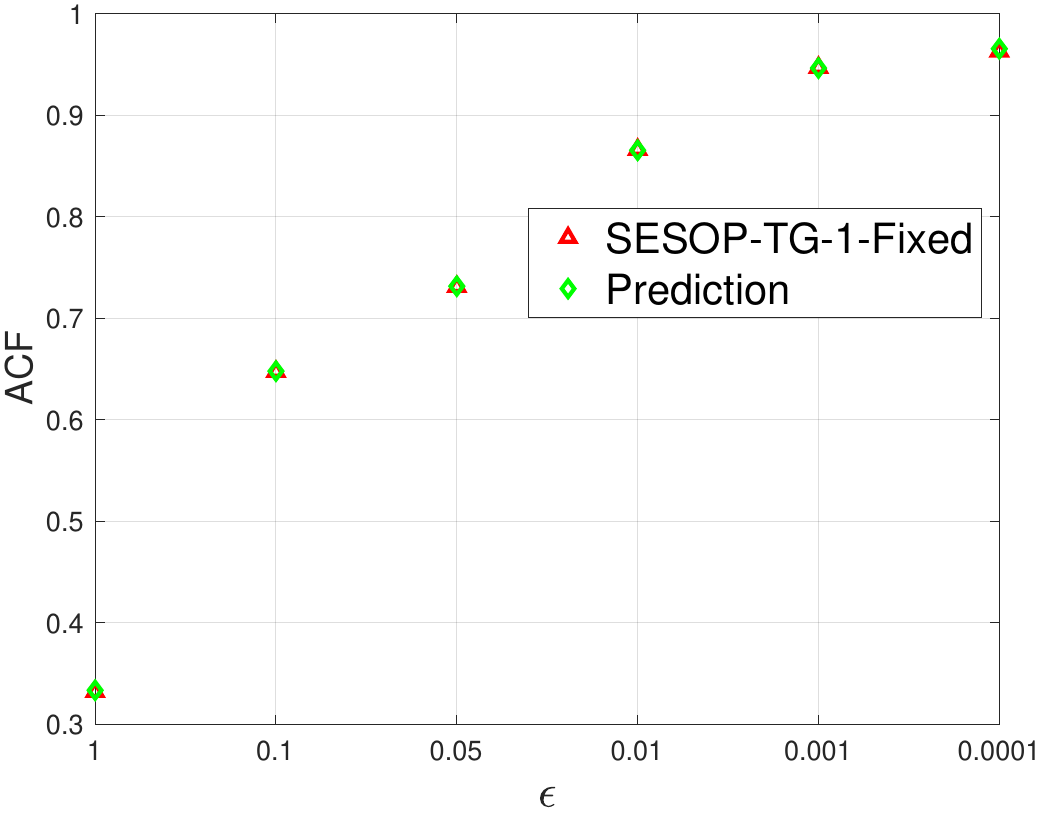}\label{fig:thore:prac:rotated:2}}\\
	\subfigure[$\phi=\frac{\pi}{5}$]{\centering \includegraphics[width=0.45\textwidth]{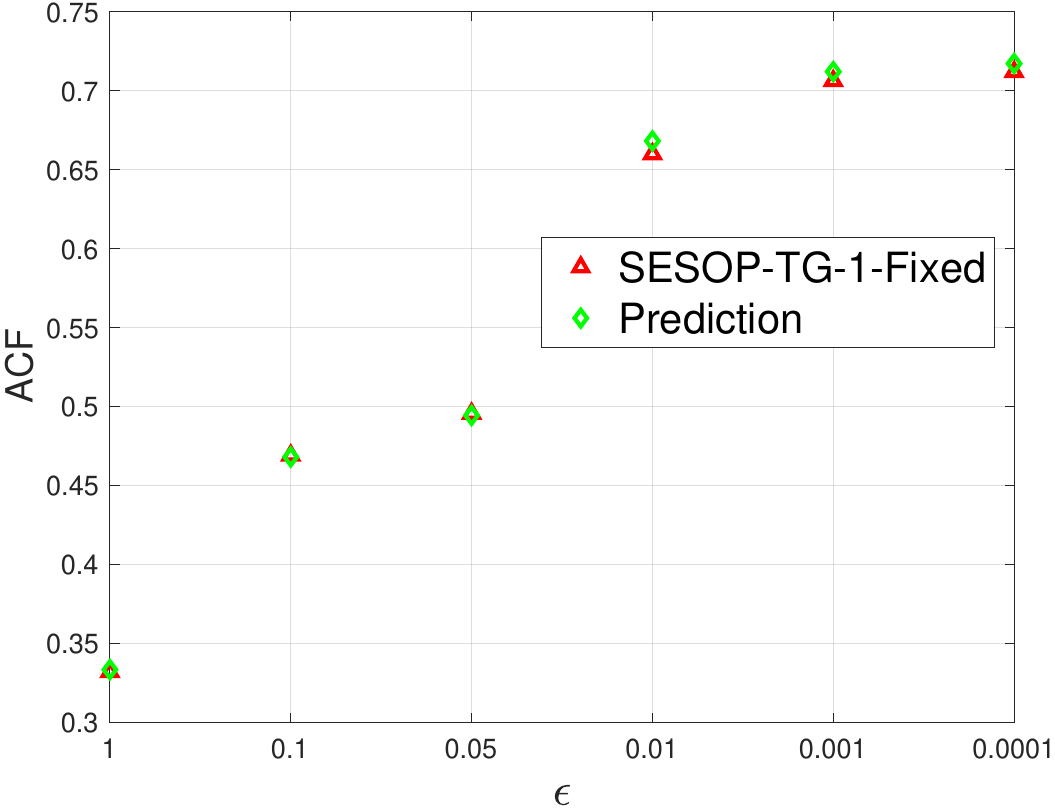}\label{fig:thore:prac:rotated:3}}
	\subfigure[$\phi=\frac{\pi}{4}$]{ \centering \includegraphics[width=0.45\textwidth]{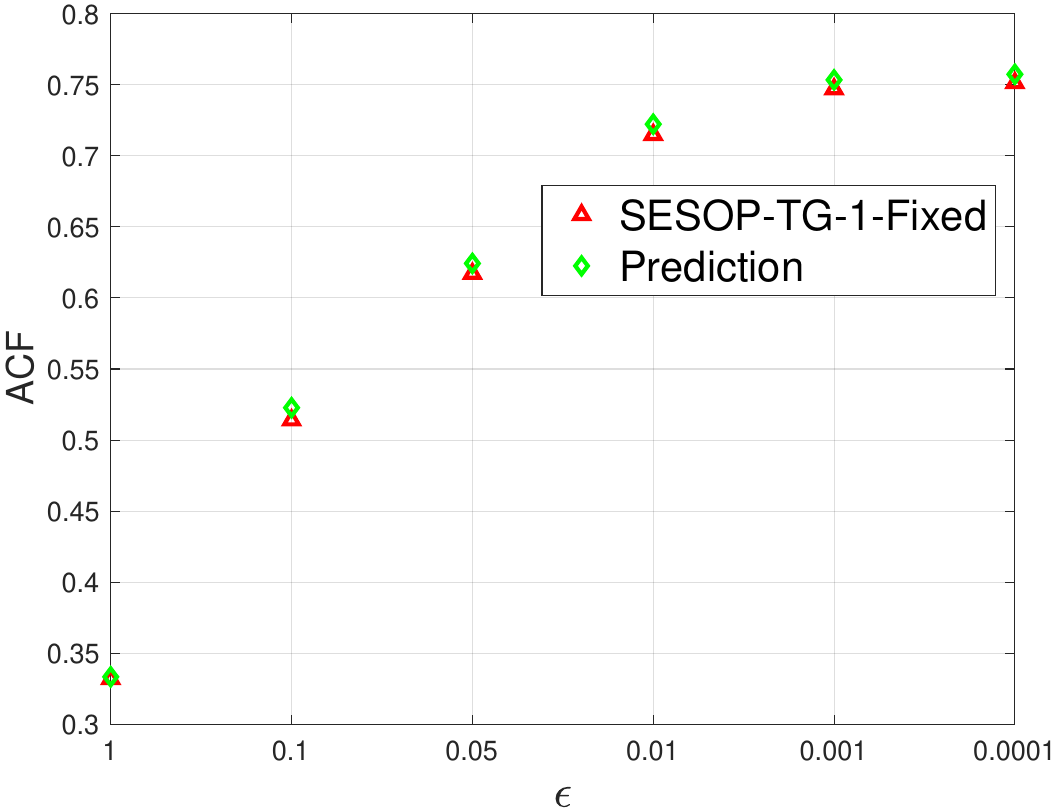}\label{fig:thore:prac:rotated:4}}
	\caption{Comparison of the predicted ACF to the convergence factor achieved in practice.}\label{fig:thore:prac:rotated}
\end{figure}

\subsection*{Comparison between subspace minimization and optimal fixed stepsizes} 
Next, we compare the ACFs of three different approaches to determining the fixed stepsizes: 1) SESOP-TG-$1$-Fixed defined above; 2) subspace minimization (classical SESOP); 3) optimized fixed stepsizes. The optimized stepsizes are obtained by minimizing the true condition number $\kappa$ of $\mA_\alpha$ in \eqref{eq:A_Alpha} by employing linesearch over $\alpha$. The test problem remains unchanged, except that we test both bilinear and bicubic prolongations and use $64\times64$ grids. The ACF achieved in practice is estimated as above by the geometric mean of the last $5$ iterations when the algorithm reaches $500$ iterations or the residual norm is smaller than $10^{-8}$. Note also that $\kappa_{opt}$ of \Cref{sec:connection_h_Ellipticity}, i.e., one over the $h$-ellipticity measure, is used here to represent the idealized ACF for comparison.

From \Cref{tab:conv:rotated:h:tg}, we clearly see that optimized fixed stepsizes yield a lower ACF than SESOP-TG-$1$-Fixed for the rotated anisotropic diffusion problem, and in fact its ACF is at least as good as that obtained by subspace minimization. This suggests that, if we can estimate the optimal stepsizes efficiently, then we can significantly reduce computation time because the subspace minimization, in general, is expensive. Also, we see that bicubic prolongation yields a lower ACF than bilinear prolongation, consistent with \cite{hemker1990order}. 

The relative disadvantage of SESOP-TG-$1$-Fixed for rotated anisotropic diffusion stems mainly from fixing $c_3 = 1$, whereas the optimal $c_3$ is higher for this problem, consistent with the analysis of \cite{yavneh1998coarse}. 

\begin{table}
\caption{Comparison of the ACF of SESOP-TG-$1$-Fixed (TG-1), classical subspace minimization (SESOP), and optimized stepsizes (Opt). The idealized convergence factor of \Cref{sec:connection_h_Ellipticity}, based on the $h$-ellipticity measure, is included as a benchmark (Idealized).}
\label{tab:conv:rotated:h:tg}
\begin{center}
\begin{tabular}{||c|c||c|c|c|c|c|c|c||} \hline\hline
\multirow{2}{*}{$\phi$} & \multirow{2}{*}{$\epsilon$} & \multicolumn{3}{c|}{Bilinear} &  \multicolumn{3}{c|}{Bicubic}& \multirow{2}{*}{Idealized} \\ \cline{3-8}
&  & TG-1 & SESOP & Opt & TG-1 & SESOP & Opt & \\ 
\hline\hline

 $0$ & $1$ & $0.333$& $0.332$ &$0.332$& $0.333$ &  $0.333$ &  $0.331$   & $0.333$       
\\ \hline

$\frac{\pi}{6}$ & $10^{-3}$ & $0.669$& $0.561$ & $0.563$ &$0.587$&  $0.533$  &  $0.532$   & $0.587$       \\ \hline 

$\frac{\pi}{6}$ & $10^{-4}$ & $0.676$& $0.563$ & $0.565$ & $0.588$ & $0.534$  & $0.533$   &  $0.588$     \\ \hline

$\frac{\pi}{4}$ & $10^{-3}$ & $0.753$& $0.500$ & $0.535$ &$0.653$&   $0.454$ & $0.443$  & $0.446$    \\ \hline

$\frac{\pi}{4}$ & $10^{-4}$ & $0.757$& $0.502$ & $0.537$ & $0.658$& $0.451$  & $0.445$   &$0.446$     \\ \hline\hline

\end{tabular}
\end{center}
\end{table}

\subsection*{Optimizing $\kappa$ of $\mA_\alpha$ in practice} We next study the efficacy of the two heuristic methods proposed in \Cref{sec:sub:optCondPractice}---estimating $\kappa_{opt}$ on a coarse grid or using \eqref{eq:alphaoptSmoothAnaly}---to select the fixed stepsizes. 

In the following tests we use $1024\times 1024$ grids to discretize \eqref{eq:anis_diff_con:xy}. Denote by $r^{opt}_{Num}$ the ACF obtained by optimizing $\kappa$ of $\mA_\alpha$ on a $Num\times Num$ grid. Then, we define the ``Deterioration Factor'' (DF),
$$
DF(Num) \triangleqtemp \frac{\log r^{opt}_{1024}}{\log r^{opt}_{Num}} \, ,
$$
to measure the deterioration of the ACF incurred by optimizing $\kappa$ on a coarser grid. For example, $DF(Num)=2$ means that, asymptotically, it takes twice as many iterations of the algorithm which uses fixed stepsizes optimized on coarser grids to achieve the same error reduction as a single iteration with the true optimal stepsize.  Additionally, we compare between bilinear and bicubic interpolation \cite{hemker1990order}. In \Cref{fig:lowhighresoDiff}, we see $DF(Num)\approx 1.05$ for $Num\geq 128$, which means that determining the stepsizes on $128\times 128$ grids results in just a $\sim5\%$ asymptotic increase in the required iterations. Moreover, we observe that the use of higher order interpolation tends to reduce the DF.

\begin{figure}[!htb]
	\centering
	\subfigure[$\epsilon=10^{-3}$.]{ \centering\includegraphics[width=0.48\textwidth]{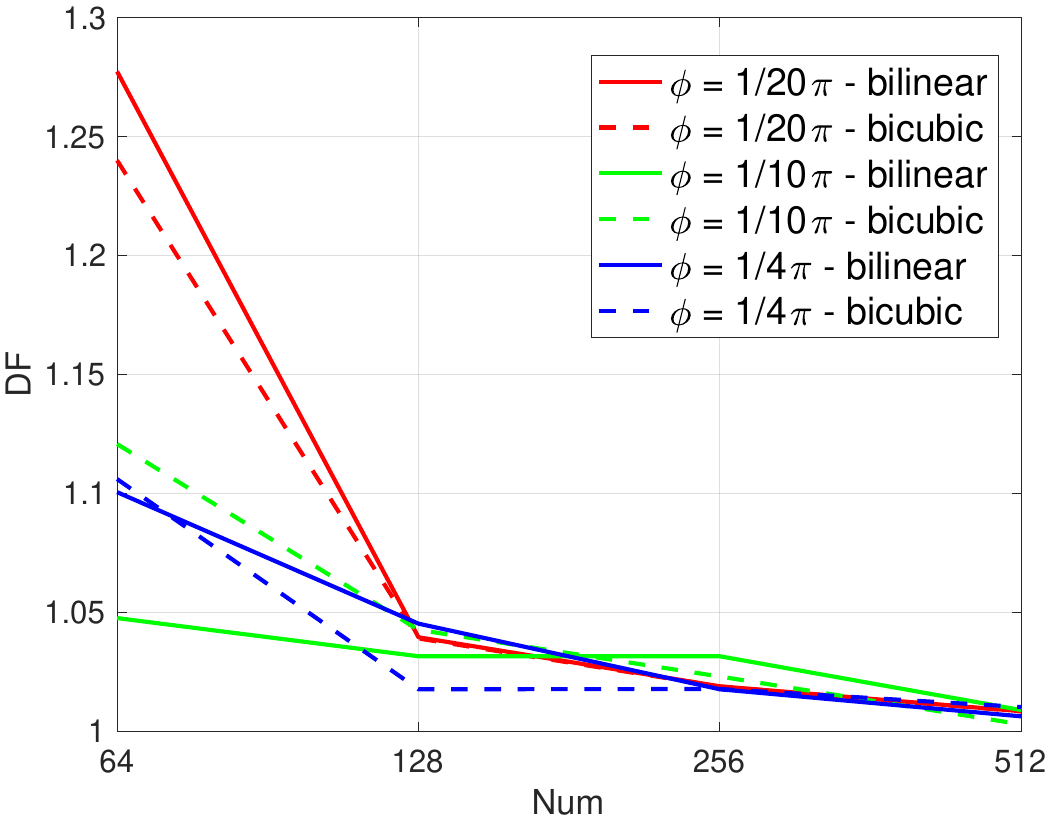}}
    \subfigure[$\phi=\frac{\pi}{30}$.]{ \centering\includegraphics[width=0.48\textwidth]{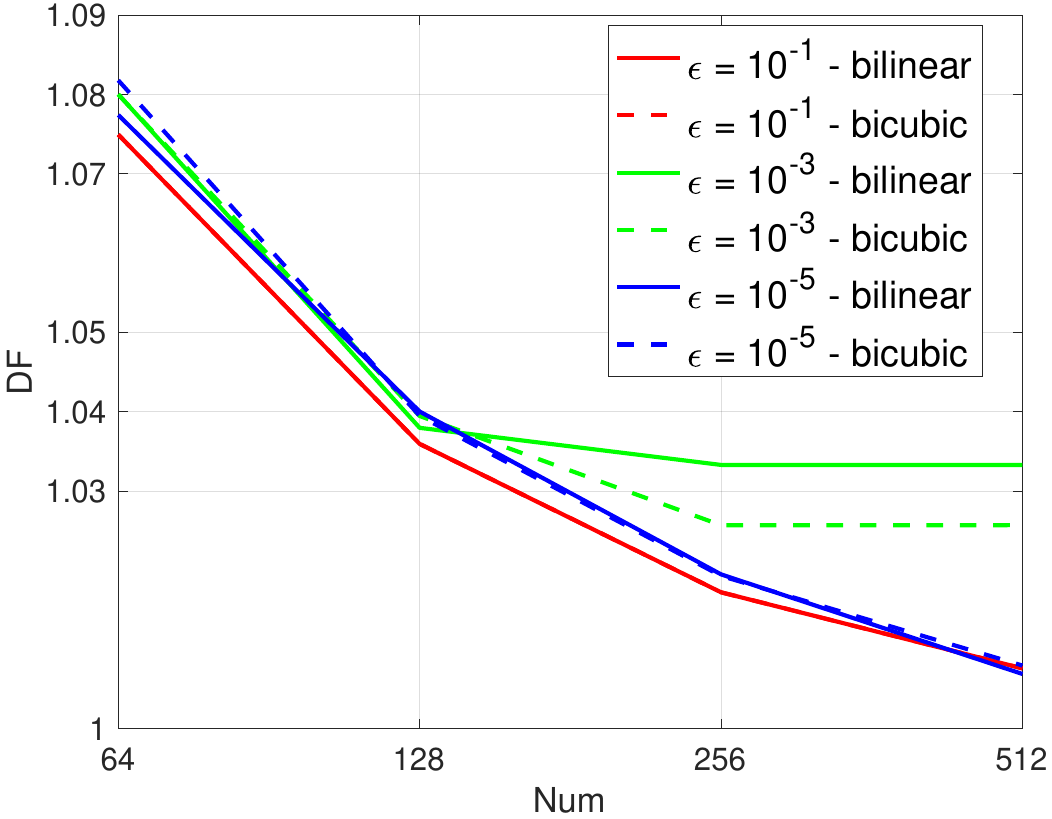}}
	\caption{$DF(Num)$ for various $\epsilon$ and $\phi$.}\label{fig:lowhighresoDiff}
\end{figure}

\begin{figure}[!htb]
    \centering
    \subfigure[$\epsilon=10^{-3},~\phi=\frac{\pi}{4}$]{\includegraphics[width=0.49\textwidth]{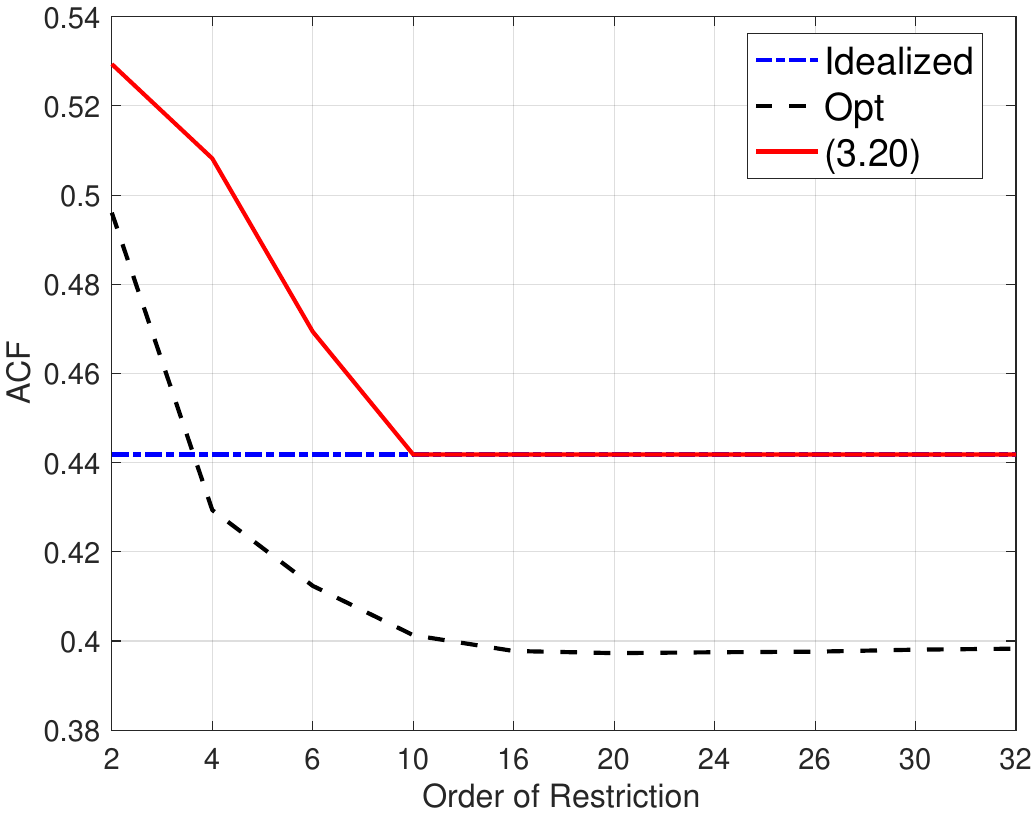}}
     \subfigure[$\epsilon=10^{-3},~\phi=\frac{2\pi}{5}$]{\includegraphics[width=0.49\textwidth]{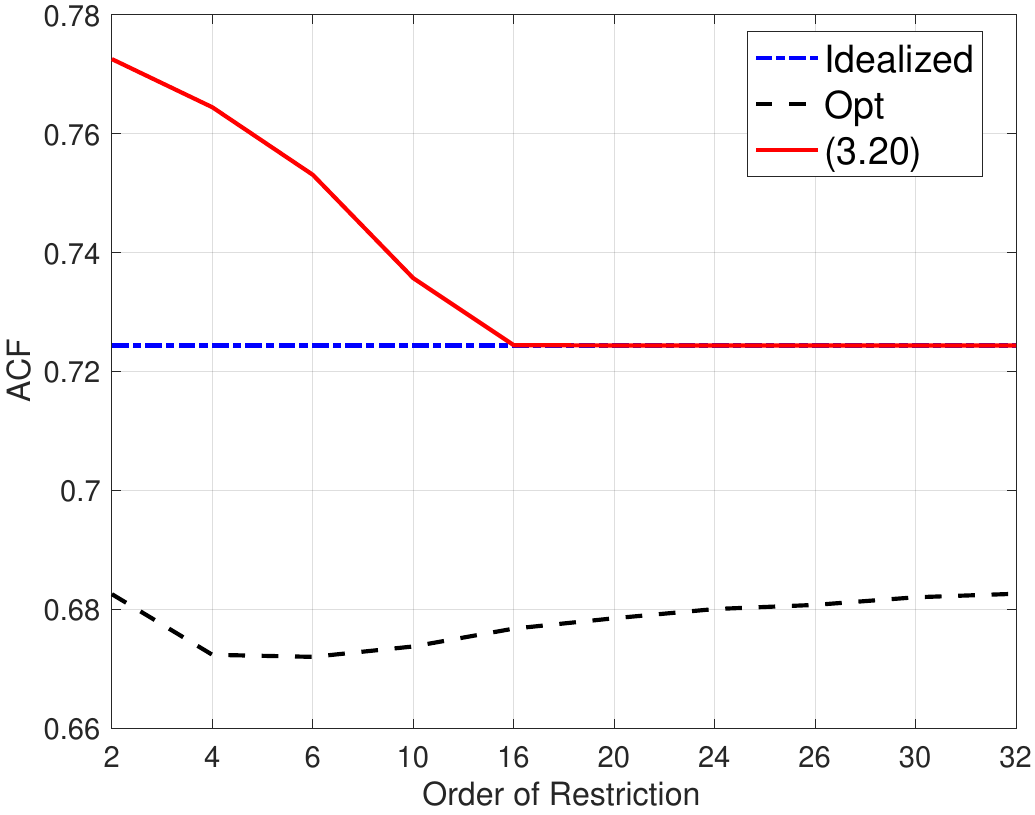}}
     
     \subfigure[$\epsilon=10^{-3},~\phi=\frac{\pi}{4}$]{\includegraphics[width=0.49\textwidth]{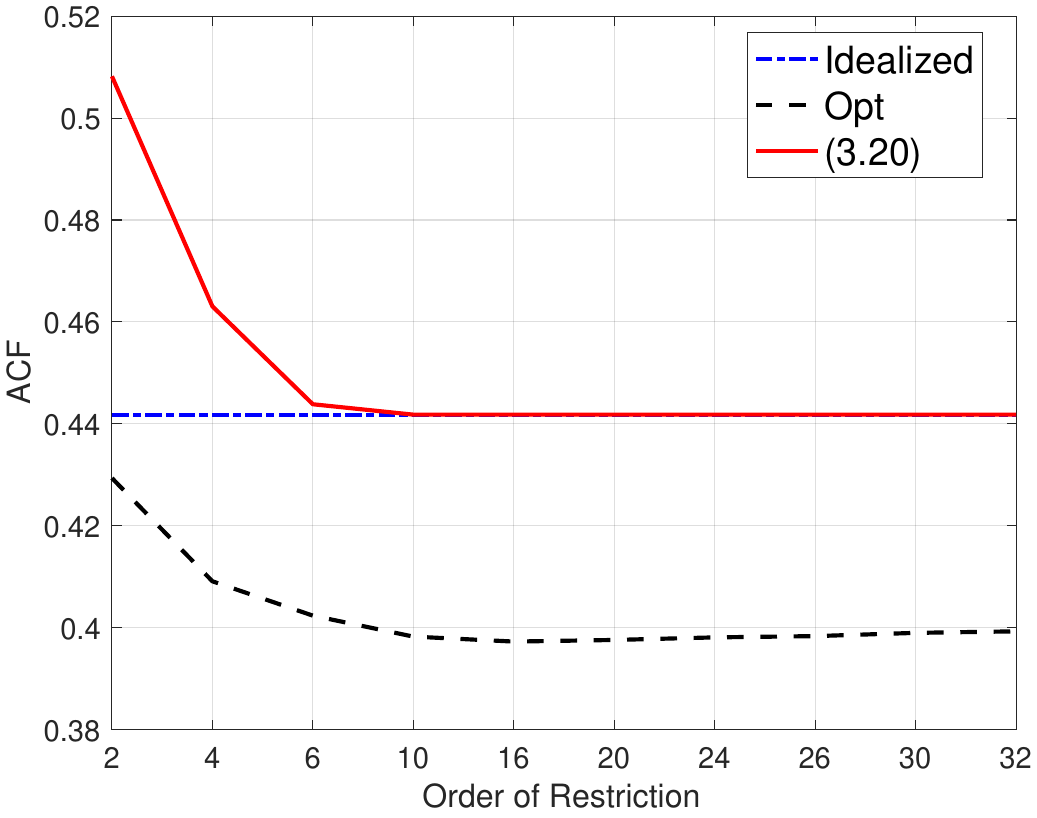}}
     \subfigure[$\epsilon=10^{-3},~\phi=\frac{2\pi}{5}$]{\includegraphics[width=0.49\textwidth]{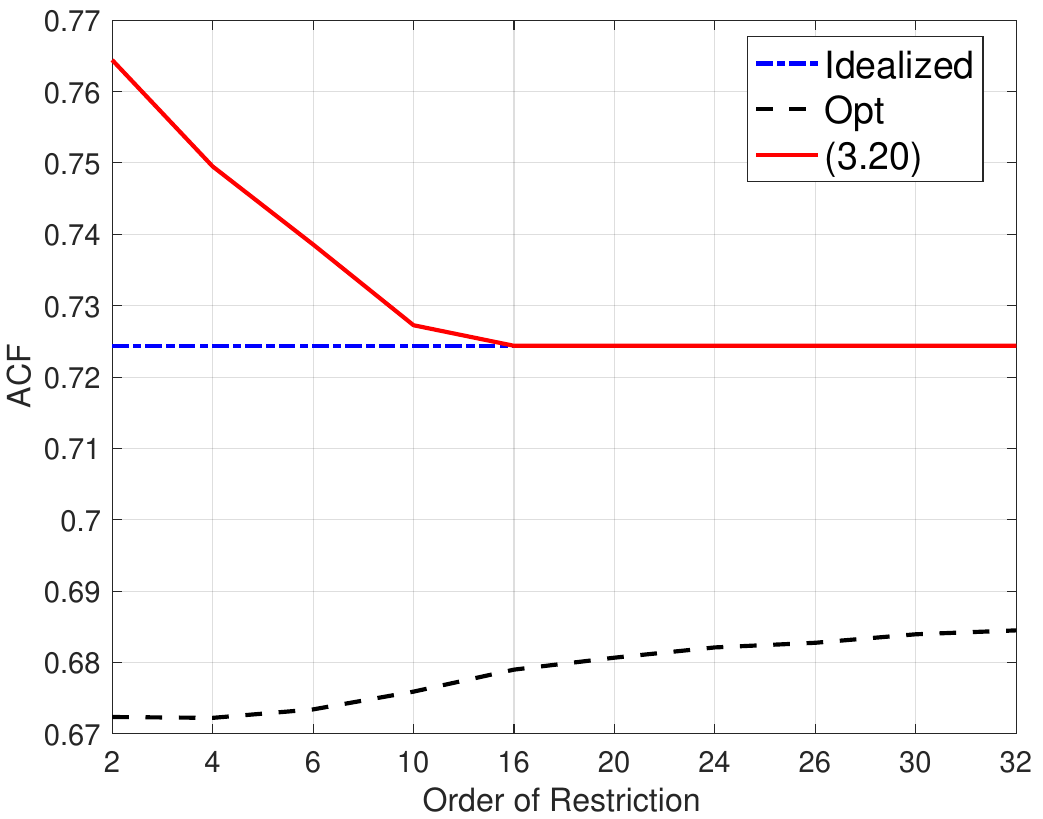}}
    \caption{Comparison of ACFs for varying order of intergrid transfers, using $512\times 512$ grids and Galerkin coarsening. First row: bilinear interpolation; Second row: bicubic interpolation. ``Idealized'' and ``Opt'' refer to \Cref{tab:conv:rotated:h:tg}.} 
    \label{fig:estimatedstepsizes}
\end{figure}


Now, we study the accuracy of using \eqref{eq:alphaoptSmoothAnaly} for selecting the fixed stepsizes.   From \cite[Equation (14)]{yavneh1998coarse}, we know that the accuracy of \eqref{eq:KappaSmoothAnalysis} for approximating $\kappa$ of $\mA_\alpha$ depends on the order of intergrid transfers. In \Cref{fig:estimatedstepsizes}, we show a comparison of ACFs with different orders of intergrid transfers. Note that the intergrid transfers used here have the same order for low and high-frequency modes. From \Cref{fig:estimatedstepsizes}, we find that \eqref{eq:alphaoptSmoothAnaly} becomes more accurate and very close to the idealized estimate when the order of restriction increases. This is due to the fact that higher order intergrid transfer operators filter out the high-frequency modes, and therefore the idealized assumptions are more closely satisfied. 
Note that the optimized algorithm is even better than the the idealized one, as also observed in \Cref{tab:conv:rotated:h:tg}, because the optimized algorithm takes into account all the modes globally when minimizing $\kappa$. Moreover, for the optimized version we see that increasing the order of the restriction begins increase the ACF slightly, bringing it closer to the idealized value. We finally note that the purpose of this test is to academically study the relationship between the accuracy of \eqref{eq:alphaoptSmoothAnaly} and the use of intergrid transfers and, in practice, it is not cost-effective to use very high order intergrid transfers.

\begin{remark}
	\label{remark:LFA_smoothingAnalysis:compare}
In this part we have examined the accuracy of two heuristic methods for determining the fixed stepsizes. In practice, the method shown in \Cref{fig:lowhighresoDiff} may be more attractive because the increase of iterations due to computing the stepsizes on a coarse grid is modest. However, for some practical problems,  \eqref{eq:alphaoptSmoothAnaly} is also attractive because its computation is cheaper than working on a coarse grid. Moreover, for many practical problems, we need to solve \eqref{eq:discretization:linear} with multiple $\vf^h$, and then the additional computation for selecting the stepsizes is negligible. 
\end{remark} 

\subsection*{Practical tests}
Now we study the performance of applying the multilevel version (V-cycle) of the proposed scheme to \eqref{eq:anis_diff_con:xy}. The multilevel version of  \Cref{alg:TG} is denoted by ``MG''. Using preconditioned conjugate gradients (PCG) with MG as the preconditioner is denoted by ``PCG-MG''. The fixed-stepsize version of SESOP-MG-$1$ is denoted by SESOP-MG-$1$-Opt when $\alpha_{opt}$ is computed on a coarse grid and by SESOP-MG-$1$-\eqref{eq:alphaoptSmoothAnaly} when \eqref{eq:alphaoptSmoothAnaly} $\alpha_{opt}$ is used. We use $1024\times 1024$ grids to discretize  \eqref{eq:anis_diff_con:xy} and obtain $\alpha_{opt}$ for SESOP-MG-$1$-Opt on $128\times 128$ grids. The coarse problems are obtained by direct rediscretization, and the bilinear prolongation and full-weighted residual transfers are employed. For coarse problems, we use Jacobi relaxation with optimal damping factor as estimated in \cite{yavneh1998multigrid} and $\nu_1=2$ and $\nu_2=1$. Note that on the finest level, we set $\nu_1=0$ and $\nu_2=0$ and only one evaluation of the gradient is allowed. 
 
 In \Cref{fig:practical test:examI:resi} ($\epsilon=1,\,\phi=0$), we see SESOP-MG-$1$ and PCG-MG perform identically. However, in \Cref{fig:practical test:examII:resi} ($\epsilon=10^{-3},\,\phi=\frac{\pi}{4}$), we find that SESOP-MG-$1$ becomes faster than PCG-MG. Indeed, in this case, it helps to scale the residual before restricting to the coarse for MG (also see the deterioration of MG in \Cref{fig:practical test:examII:resi}), but SESOP-MG-$1$ can automatically scale the residual via the subspace minimization \cite{brandt2011multigrid}. Moreover, from \Cref{fig:practical test:examI:resiCPU,fig:practical test:examII:resicpu}, we observe that SESOP-MG-$1$-Opt and SESOP-MG-$1$-\eqref{eq:alphaoptSmoothAnaly} are the fastest methods in terms of CPU time because these two methods inherit the effectiveness of SESOP-MG-$1$ but need much less CPU time. It is interesting to note that SESOP-MG-$1$-Opt and SESOP-MG-$1$-\eqref{eq:alphaoptSmoothAnaly} work almost the same in these two tests, demonstrating the effectivenss of the proposed strategies for determining the stepsizes. 

\begin{figure}[!htb]
\centering
	\subfigure[$\epsilon=1,\,\phi=0$]{\includegraphics[width=0.49\textwidth]{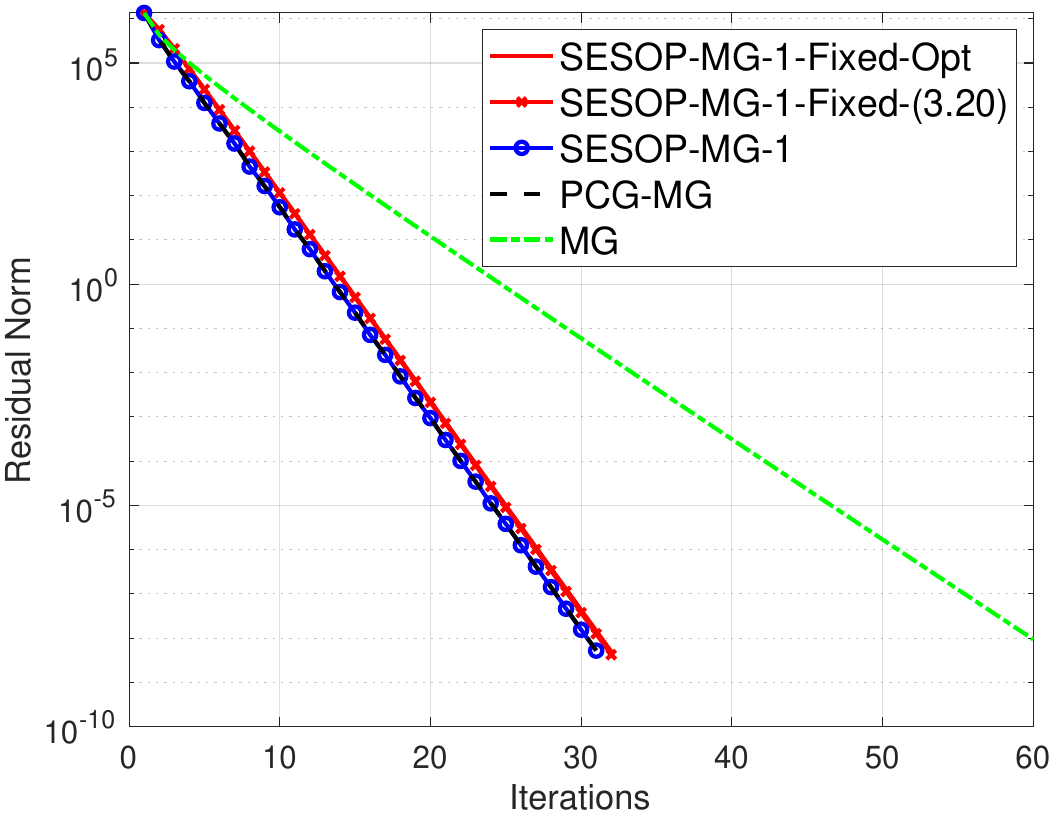}\label{fig:practical test:examI:resi}}
	\subfigure[$\epsilon=1,\,\phi=0$]{\includegraphics[width=0.49\textwidth]{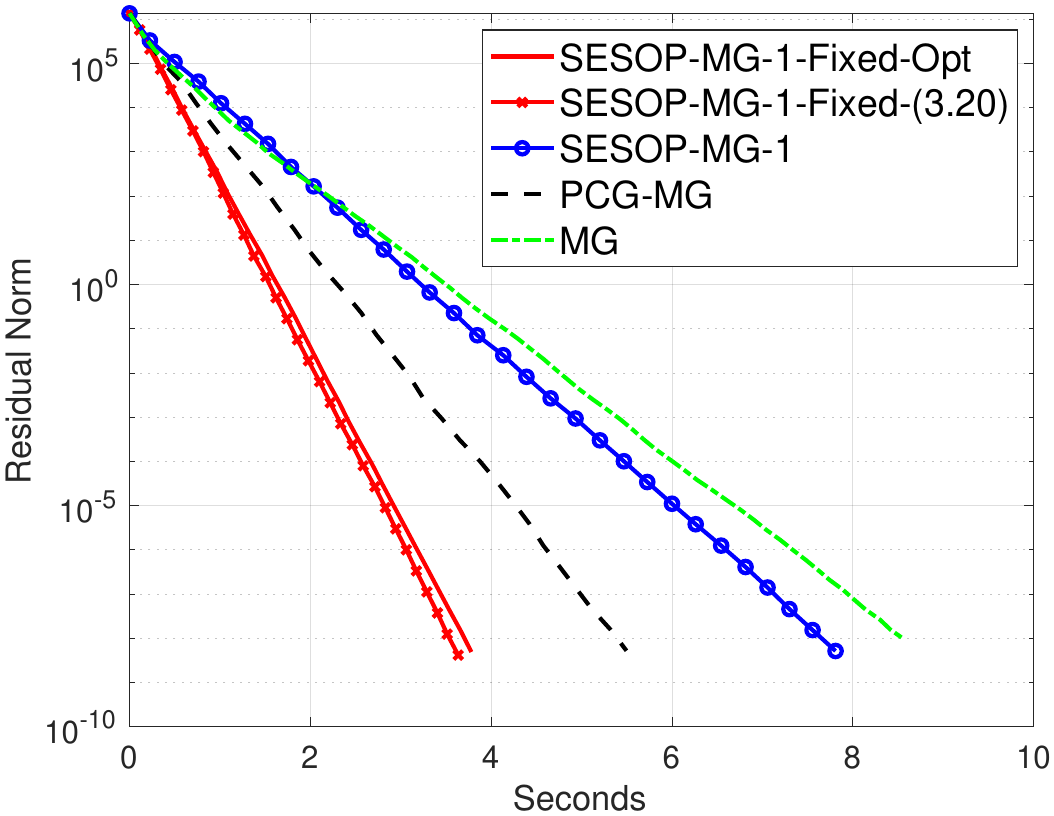}\label{fig:practical test:examI:resiCPU}}
	
		\subfigure[$\epsilon=10^{-3},\,\phi=\frac{\pi}{4}$]{\includegraphics[width=0.49\textwidth]{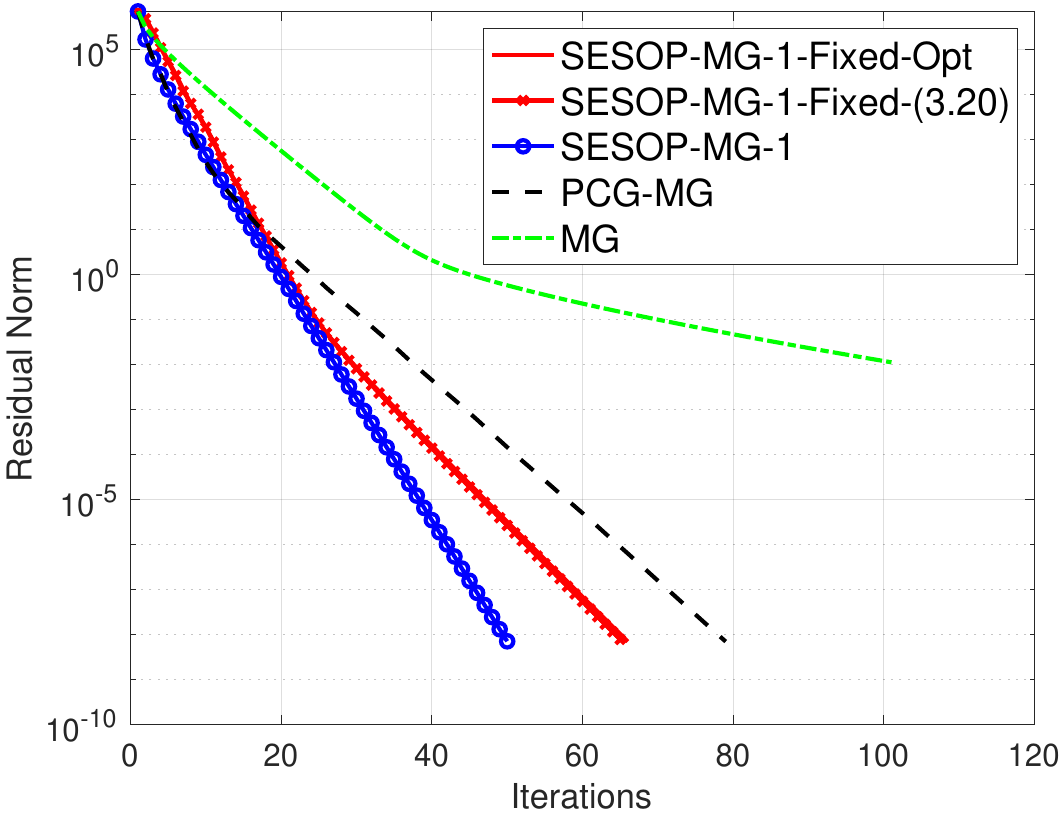}\label{fig:practical test:examII:resi}}
	\subfigure[$\epsilon=10^{-3},\,\phi=\frac{\pi}{4}$]{\includegraphics[width=0.49\textwidth]{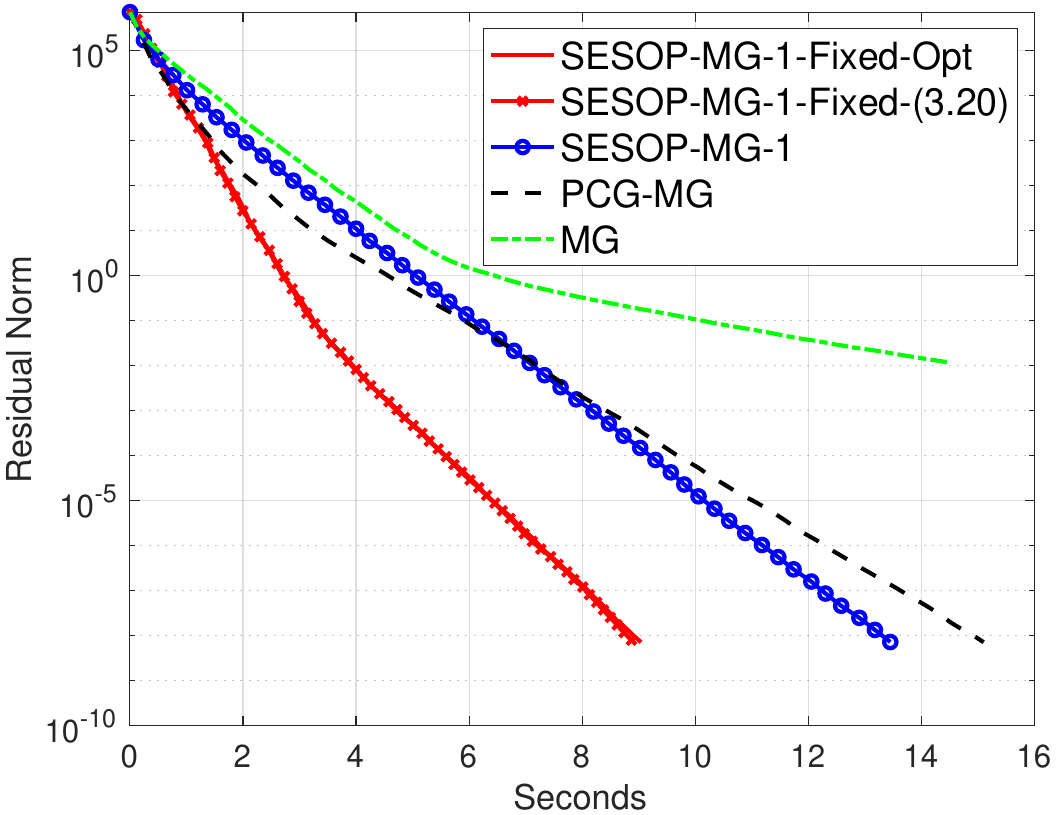}\label{fig:practical test:examII:resicpu}}
	\caption{Comparison of different methods for \eqref{eq:anis_diff_con:xy}. Test on $1024\times 1024$ grids and use the V-cycle.}
	\label{fig:practical test}
\end{figure}

\section{Conclusion}
\label{sec:conclusion}
In this paper, we merge multigrid (MG) optimization with SESOP optimization. The numerical experiments on linear and nonlinear problems illustrate the effectiveness and robustness of our scheme. Moreover, for linear problems, if only one history is used, we derive optimal fixed stepsizes that allow us to avoid the expensive subspace minimization and save computation. Specifically, for elliptic PDEs with constant coefficients, we propose two heuristic methods to estimate the fixed stepsizes based on LFA and smoothing analysis and the efficiency of these two methods is demonstrated in numerical tests.

\appendix

\section{Proof of \texorpdfstring{\Cref{lemma_1:lar:sma:ACF}}{TEXT}}\label{proof:lemma_1:lar:sma:ACF}
Consider $\hat r(c_1,c_{23},a_{\alpha})$ in \eqref{eq:worst_conv_rate:def} as a continuous function of a positive 
variable $a_{\alpha}\in \left[ a_{\alpha}^{min},a_{\alpha}^{max}\right]$, with fixed $c_1$ and $c_{23}$: 
$$
\hat r(c_1,c_{23},a_{\alpha})=\frac{1}{2}\left|b(c_1,c_{23},a_{\alpha})+\text{sgn}\left(b(c_1,c_{23},a_{\alpha})\right)\sqrt{b^2(c_1,c_{23},a_{\alpha})-4c_1}\right|.
$$
We distinguish between two regimes: (I): $b^2(c_1,c_{23},a_{\alpha}) < 4c_1$,  the square-root term is imaginary, and we simply get  $\hat r(c_1,c_{23},a_{\alpha})=\sqrt{c_1}$; (II): $b^2(c_1,c_{23},a_{\alpha}) \geq 4c_1$,the square-root term is real and the derivative of $\hat r(c_1,c_{23},a_{\alpha})$ with respect to $a_{\alpha}$ is
$$
\frac{\partial \hat r(c_1,c_{23},a_{\alpha})}{\partial a_{\alpha}}=-\frac{c_{23}}{2}\text{sgn}\left(b(c_1,c_{23},a_{\alpha})\right)\left[1+\frac{\left|b(c_1,c_{23},a_{\alpha})\right|}{\sqrt{b^2(c_1,c_{23},a_{\alpha})-4c_1}}\right].
$$
We ignore the irrelevant choice $c_{23} = 0$, for which the method is obviously not convergent. Using $b(c_1,c_{23},a_{\alpha}) = -c_{23}\left(a_{\alpha} -\frac{1+c_1}{c_{23}}\right)$, we get 
$$
\frac{\partial \hat r(c_1,c_{23},a_{\alpha})}{\partial a_{\alpha}}=\frac{|c_{23}|}{2}\text{sgn}\left(a_{\alpha} -\frac{1+c_1}{c_{23}}\right)\left[1+\frac{\left|b(c_1,c_{23},a_{\alpha})\right|}{\sqrt{b^2(c_1,c_{23},a_{\alpha})-4c_1}}\right].
$$
Notice that $\hat r(c_1,c_{23},a_{\alpha})$ is a symmetric function of $a_{\alpha}-(1+c_1)/c_{23}$ and it is furthermore convex because its derivative is strictly negative for $a_{\alpha} < (1+c_1)/c_{23}$ and positive for $a_{\alpha} > (1+c_1)/c_{23}$. It follows that, regardless of the sign of $b^2(c_1,c_{23},a_{\alpha})-4c_1$ throughout the regime $a_{\alpha} \in [a_{\alpha}^{min},a_{\alpha}^{max}]$, there exists no local maximum of $\hat r(c_1,c_{23},a_{\alpha})$. Hence,  $ \bar r(c_1,c_{23})=\max\left(\hat r(c_1,c_{23},a_{\alpha}^{min}),\hat r(c_1,c_{23}, a_{\alpha}^{max})\right).$ \qedsymbol
\section{Proof of \texorpdfstring{\Cref{lemma_2:determ:c_23}}{TEXT}}\label{proof:lemma_2:determ:c_23}
Consider $\hat r(c_1,c_{23},a_{\alpha})$ as a continuous function of a real variable $c_{23}$ with $c_1$ and $a_{\alpha}$ fixed. For $b^2(c_1,c_{23},a_{\alpha}) \geq 4c_1$, the derivative of $\hat r(c_1,c_{23},a_{\alpha})$ with respect to $c_{23}$  is
$$
\frac{\partial \hat r(c_1,c_{23},a_{\alpha})}{\partial c_{23}}=\frac{|a_{\alpha}|}{2}\text{sgn}\left(c_{23} -\frac{1+c_1}{a_{\alpha}}\right)\left[1+\frac{\left|b(c_1,c_{23},a_{\alpha})\right|}{\sqrt{b^2(c_1,c_{23},a_{\alpha})-4c_1}}\right].
$$
As in \Cref{lemma_1:lar:sma:ACF}, the meeting point of $\hat r(c_1, c_{23},a_{\alpha}^{max})$ and $\hat r(c_1, c_{23},a_{\alpha}^{min})$, which lies between $(1+c_1)/a_{\alpha}^{max}$ and $(1+c_1)/a_{\alpha}^{min}$, is where $\hat r(c_1,c_{23},a_{\alpha})$ is minimized with respect to $c_{23}$ and then the optimal $c_{23}$ enforces $\hat r(c_1, c_{23},a_{\alpha}^{max})=\hat r(c_1, c_{23},a_{\alpha}^{min})$ resulting in $1+c_1-c_{23}a_{\alpha}^{max} = -(1+c_1-c_{23}a_{\alpha}^{min}),$
leading to $c_{23}^* = \frac{2\left(1+c_1\right)}{a_{\alpha}^{min}+a_{\alpha}^{max}}.$  For $b^2(c_1,c_{23},a_{\alpha}) < 4c_1$, we have $\hat r(c_1,c_{23},a_{\alpha})=\sqrt{c_1}$ which is irrelevant to $c_{23}$. \qedsymbol
\section{Proof of \texorpdfstring{\Cref{lemma_3:determ:c_1}}{TEXT}}\label{proof:lemma_3:determ:c_1}
Consider two regimes: (I) $\mu^2(1+c_1)^2\leq 4c_1$; (II) $\mu^2(1+c_1)^2\geq 4c_1$. For (I), we have $\bar r_{c_{23}^*}(c_1)=\sqrt{c_1}$ and $c_1\in\left[c_1^-,c_1^+\right]$ where $c_1^\pm=\left(2-\mu^2\pm\sqrt{4-4\mu^2}\right)/\mu^2$ are the two solutions of $\mu^2(1+c_1)^2 = 4c_1$. Evidently, $c_1^-$ is the solution to minimize $r_{c_{23}^*}(c_1)$.

 For (II), we have $c_1\leq c_1^-~\text{or}~c_1\geq c_1^+$. We first note that $$\mu (1+c_1)+\sqrt{\mu^2(1+c_1)^2-4c_1}>0.$$ and then $\bar r_{c_{23}^*}(c_1)=\frac{1}{2}\left[\mu (1+c_1)+\sqrt{\mu^2(1+c_1)^2-4c_1}\right]$. Evidently, the derivative of $\bar r_{c_{23}^*}(c_1)$ with respect to $c_1$ is 
\e
\frac{\partial \bar r_{c_{23}^*}(c_1)}{\partial c_1}=\frac{1}{2}\left(\mu+\frac{\mu^2(1+c_1)-2}{\sqrt{\mu^2(1+c_1)^2-4c_1}}\right).\label{eq:worst_conv_rate:dete:c_1:derivative}
\ee
Note that $\frac{\mu^2(1+c_1)-2}{\sqrt{\mu^2(1+c_1)^2-4c_1}}$ is positive for $c_1>\frac{2}{\mu^2}-1$ and negative for $c_1<\frac{2}{\mu^2}-1$, hence, \eqref{eq:worst_conv_rate:dete:c_1:derivative} is positive for $c_1>c_1^+$, whereas for $c_1<c_1^-$ we have
$$
\begin{array}{rcl}
\frac{\partial \bar r_{c_{23}^*}(c_1)}{\partial c_1}&=&\frac{1}{2}\left(\mu-\sqrt{\frac{\left(\mu^2(1+c_1)-2\right)^2}{\mu^2(1+c_1)^2-4c_1}}\right)=\frac{\mu}{2}\left(1-\sqrt{\frac{\mu^4(1+c_1)^2-4\mu^2(1+c_1)+4}{\mu^4(1+c_1)^2-4\mu^2c_1}}\right)\\
&=&\frac{\mu}{2}\left(1-\sqrt{1+\frac{4(1-\mu^2)}{\mu^4(1+c_1)^2-4\mu^2c_1}}\right)<0
\end{array}
$$
The last inequality is due to the fact that that $0<\mu<1$ and $\mu^2(1+c_1)^2-4c_1>0$. Evidently, the optimal $c_1$ to minimize $\bar r_{c_{23}^*}(c_1)$ is either $c_1^-$ or $c_1^+$. However, $\bar r_{c_{23}^*}(c_1)>1$ for $c_1^+$ that $c_1^-=\left(2-\mu^2-\sqrt{4-4\mu^2}\right)/\mu^2$ is the solution. Substituting the expression of $\mu$ into $c_1^-$, we get the desired result.
\qedsymbol

\section{Proof of \texorpdfstring{\Cref{lem:ConditionNumber}}{TEXT}}\label{proof:lemma_3:idealProRest}
Any eigenvector $\va_i \in R(\mI_H^h)$ (respectively, $\va_i \notin \mathcal R(\mI_H^h))$) is an eigenvector of the CGC matrix $\mI_H^h\mA_H^{-1}(\mI_H^h)^\mathcal T \mA$, with eigenvalue 1 (respectively, 0), because if $\va_i \in \mathcal R(\mI_H^h)$ then it can be written as $\va_i = \mI_H^h \ve_j$ (that is, it is the $j$th column for some $j$), so    
$$
\left[ \mI_H^h\mA_H^{-1}(\mI_H^h)^\mathcal T \mA \right] \va_i = \mI_H^h\left[(\mI_H^h)^\mathcal T \mA \mI_H^h\right]^{-1}\left[(\mI_H^h)^\mathcal T \mA \mI_H^h \right] \ve_j = \mI_H^h \ve_j = \va_i,
$$
whereas if $\va_i \notin \mathcal R(\mI_H^h)$ then it is orthogonal to the columns of $\mI_H^h$, so
$$
\left[ \mI_H^h\mA_H^{-1}(\mI_H^h)^\mathcal T \mA \right] \va_i = \left[ \mI_H^h\mA_H^{-1}(\mI_H^h)^\mathcal T \right] a_i \va_i = \bm 0.
$$
It follows that the eigenvectors of $\mA_\alpha$ are $\va_i$, with eigenvalues given by 
$$
a_{\alpha}^i =  \left\{ \begin{array}{ll} \alpha a_i + 1 - \alpha & {\rm if}~ \va_i \in \mathcal R(\mI_H^h), \\ \alpha a_i & {\rm otherwise}. \end{array} \right. 
$$
By using the definition of $\kappa$, $a_{fmax}$, $a_{cmax}$, $a_{fmin}$, and $a_{cmin}$, the desired result is derived. \qedsymbol
\section{Proof of \texorpdfstring{\Cref{theorem:OptimalAlpha}}{TEXT}}\label{proof:theoremOptalpha:idealProRest}
To minimize $\kappa$, the optimal $\alpha$ should minimize the numerator and maximize the denominator. Note that $\alpha_{top}=\frac{1}{1+ a_{fmax}- a_{cmax}}$ and $\alpha_{bot}=\frac{1}{1+ a_{fmin}- a_{cmin}}$ are the ones to  minimize the numerator and  maximize the denominator, respectively. For $a_{fmax} -  a_{fmin} \geq  a_{cmax} -  a_{cmin}$, we have $\alpha_{top}\leq \alpha_{bot}$ and then:
\begin{enumerate}
\item
$\alpha \in(0,\alpha_{top}) \Rightarrow \kappa = \frac{\alpha  a_{cmax}+1-\alpha}{\alpha a_{fmin}},$ $\frac{d \kappa}{d \alpha} = -\frac{1}{\alpha^2 a_{fmin}} < 0,~\text{choosing}~\alpha =\alpha_{top}.$
\item
$\alpha\in(\alpha_{top},\alpha_{bot}) \Rightarrow \kappa = \frac{a_{fmax}}{a_{fmin}},~\text{choosing~any}~\alpha\in(\alpha_{top},\alpha_{bot}).$
\item
$\alpha\in(\alpha_{bot},1) \Rightarrow \kappa = \frac{\alpha  a_{fmax}}{\alpha a_{cmin}+1-\alpha}$, $\frac{d \kappa}{d \alpha} = \frac{ a_{fmax}}{(\alpha  a_{cmin}+1-\alpha)^2} > 0,~\text{choosing}~\alpha =\alpha_{bot}.$
\end{enumerate}
Evidently, the optimal $\alpha$ can be any one between $\alpha_{top}$ and $\alpha_{bot}$ yielding $\kappa_{opt}=\frac{a_{fmax}}{a_{fmin}}$. By applying the same reasoning to $a_{fmax} - a_{fmin} < a_{cmax} - a_{cmin}$, the optimal $\alpha$  achieves at $\alpha_{opt}\triangleqtemp\alpha_{bot}$. Summarizing, $\alpha_{bot}$ is the optimal solution for both cases. Substituting $\alpha_{opt}=\alpha_{bot}$ into \eqref{eq:Kappa}, we get the desired result.

\qedsymbol

\bibliographystyle{siamplain}
\bibliography{references}

\end{document}